\documentclass[a4paper,10pt]{amsart}
\usepackage[mathscr]{eucal}

\newtheorem{theorem}{Theorem}
\newtheorem{lemma}{Lemma}
\newtheorem{proposition}{Proposition}

\theoremstyle{definition}
\newtheorem{definition}{Definition}

\theoremstyle{remark}
\newtheorem{remark}{Remark}
\newtheorem{example}{Example}
\numberwithin{equation}{section}

\def\Var{\text{Var}}
\def\Cov{\textnormal{Cov}}

\def\EE{\mathbb{E}}
\def\l{\lambda_1}
\def\b{\beta}
\def\g{\gamma}

\newcommand{\Z}{\mathbb Z}
\newcommand{\R}{\mathbb R}

\begin{document}

\title[On Rates of Decay of Stationary ARCH($\infty$) models]{Long run behaviour of the
autocovariance function of ARCH($\infty$) models}
\author{John A. D. Appleby}
\address{Edgeworth Centre for Financial Mathematics, School of Mathematical Sciences, Dublin City University, Dublin 9,
Ireland}
\email{john.appleby@dcu.ie}
\urladdr{http://webpages.dcu.ie/\textasciitilde applebyj}
\author{John A. Daniels}
\address{Edgeworth Centre for Financial Mathematics, School of Mathematical Sciences, Dublin City University, Dublin 9,
Ireland} \email{john.daniels2@mail.dcu.ie}
\thanks{Both authors are partially funded by the Science Foundation Ireland grant 07/MI/008
``Edgeworth Centre for Financial Mathematics". The second author is
also supported by the Irish Research Council for Science,
Engineering and Technology under the Embark Initiative grant.}
\date{14 February 2012}

\begin{abstract}
The asymptotic properties of the memory structure of ARCH($\infty$)
equations are investigated. This asymptotic analysis is achieved by
expressing the autocovariance function of ARCH($\infty$) equations
as the solution of a linear Volterra summation equation and
analysing the properties of an associated resolvent equation via the
admissibility theory of linear Volterra operators. It is shown that
the autocovariance function decays subexponentially (or
geometrically) if and only if the kernel of the resolvent equation
has the same decay property. It is also shown that upper
subexponential bounds on the autocovariance function result if and
only if similar bounds apply to the kernel.
\end{abstract}

\maketitle
\section{Introduction}\label{sect:intro}

The significant influence of past data upon current and future
values of a time series is evidenced in many time series from the
physical sciences and finance, e.g. tree-ring data series, wheat
market prices (cf., e.g., Baillie \cite{baillie}) and stock market
and foreign exchange returns (cf., e.g., Ding and Granger
\cite{zdcg:1996}). The influence of past realisations may be defined
in terms of the persistence of the autocorrelations of the series,
with a stationary series whose autocorrelations decay at a
non-summable rate being referred to as a ``long memory" process.
Furthermore, the presence and application of long memory processes
in macroeconomics, asset pricing models and interest rate models is
noted in \cite{baillie} and the references contained therein.
Various properties of fractional Brownian motion are illustrated in
Mandelbrot and Van Ness \cite{bmjvn:1968}: of particular note is
that increments of fractional Brownian motion are stationary,
self--similar and can exhibit long memory.

Kirman and Teyssi\`ere \cite{akgt:2002a, akgt:2002b} give discrete time series models which
 are derived from a market which is composed of fundamental and technical analysts, these models are then shown to
 possess long memory characteristics in the differenced log returns
of price processes associated with these models, while other
features such as bubbles are demonstrated. Appleby and Krol
\cite{appkrol2} analyse the long memory properties of a linear
stochastic Volterra equation in both continuous and discrete time,
with conditions for both subexponential rates of decay and
arbitrarily slow decay rates in the autocovariance function being
characterised in terms of the decay of the kernel of the Volterra
equation. A continuous--time infinite history financial market model
is discussed in Anh et al. \cite{anh, aik}, which is a
generalisation of the classic Black-Scholes model, where
characterisations for long memory are proved. In each of \cite{anh,
aik, appkrol2} the equations studied have additive noise, so the
size of stochastic shocks are independent of the state of the
system.

A widely--employed class of discrete--time stochastic processes in
which the shock size depends on the state are the so--called ARCH
(autoregressive conditional heteroskedastic) processes. ARCH
processes are widely used and studied in financial mathematics to
characterise time varying conditional volatility as well as the
non--trivial autocovariance functions possessed by autoregressive
processes driven by additive noise. In particular, the ARCH
formulation captures well the tendency for clustering of volatility
Engle \cite{engle:1982}. Much of the work on ARCH processes concerns
processes with finite memory: if only the last $q$ values of the
process determine the dynamics, the process is termed an ARCH($q$)
process. A property of these finite--memory processes is that their
autocovariance functions decay exponentially fast in their time lag.
Therefore slow decay or long memory in an ARCH--type process can
only be achieved by considering terms from unboundedly far in the
past. This naturally leads to the study of ARCH($\infty$) processes
and in this work we study the memory properties such processes. A
standard definition given in e.g., \cite{lgpkrl:2000}, for these
processes is:
\begin{definition}  \label{def:arch}
    A random sequence $X=\{X(k), k\in\Z\}$ is said to satisfy ARCH($\infty$) equations
    if there exists a sequence of independent and identically
distributed (i.i.d.) non--negative random variables $\xi=\{\xi(k), k\in
\Z\}$ such that
\begin{gather} \label{eq:1a} \tag{AH}
X(k) = \varsigma(k)\xi(k), \quad \varsigma(k) = a + \sum_{j=1}^{\infty}b(j)X(k-j),
\end{gather}
where $a\geq0$ and $b=\{b(j), j\in\{1,2,...\}\}$ satisfies $b(j)\geq0$, for $j \in \{1,2,...\}$.
\end{definition}
ARCH($\infty$) processes were initially introduced by
Robinson \cite{robin:1991} as an alternative model when testing for serial correlation.
This process is a generalisation of the ``classical" ARCH($\infty$) process
\[
    r(k)=\sigma(k)\epsilon(k), \quad \sigma(k)^2 = \tau + \sum_{j=1}^{\infty}\phi(j) r(k-j)^2,
\]
where $\tau,\phi\geq0$ and $\epsilon$ is an i.i.d. random sequence.
Moreover \eqref{eq:1a} includes models where $r$ and $\sigma$ are
replaced by an arbitrary fractional positive powers of themselves
and the `shocks', $\epsilon$, are taken to be non-negative. The
terminology ARCH($\infty$) is justified, as an ARCH($\infty$)
process is in some sense the limit of an ARCH($q$) process as
$q\to\infty$. It can be seen, moreover that ARCH($\infty$) processes
are generalisations of the finite order ARCH and GARCH processes:
indeed the ARCH($q$) process of \cite{engle:1982}, results when
$\phi(j)=0$ for $j\geq q+1$ and the GARCH($p,q$) process of
Bollerslev~\cite{boll} may be rewritten as an ARCH($\infty$) process
with exponentially decaying weights $b$.

As attested to above, empirical findings indicate the presence of
long memory in financial and economic time series, which has
resulted in research being focused on the long memory properties of
stationary solutions of ARCH-like processes (cf., e.g., Baillie et
al. \cite{rbtbhm:1996}). Of note here are the investigations into
necessary and sufficient conditions for the existence of a weakly
stationary solution of the ARCH($\infty$) process, conducted by
Giraitis, Kokoszka, Leipus, Surgailis, and Zaffaroni
\cite{lgpkrl:2000, lgds:2002, pkrl:2000, Zaffaroni:2004}.
Moreover, these papers extensively study the autocovariance
structure and long memory properties of \eqref{eq:1a}.
Section~\ref{sec:autocov} details some of the results of
\cite{lgpkrl:2000, lgds:2002, Zaffaroni:2004} which are applicable
to the results of this article. Also in Section~\ref{sec:autocov} we highlight in particular the
importance of an underlying resolvent equation in determining the
long term memory characteristics of \eqref{eq:1a}.
Also, a Volterra
series representation of the autocovariance function is established.

The main results of this article appear in Section~\ref{sect:sub} where conditions on the data of
\eqref{eq:1a}, i.e., $a,b,\xi$, are given to describe decay rates in a class
wider than the class of hyperbolically decaying sequences considered heretofore.
Roughly speaking, for the memory, or kernel $b$, lying in a class of
slowing decaying (subexponential) sequences it is shown that the
autocovariance function must decay at precisely the rate of $b$.
Furthermore, we prove for the first time converse results which show
that such exact non--exponential rates of decay of the
autocovariance function result only when $b$ lies in this class.
These results strengthen the hypotheses of \cite[Theorem~2]{Zaffaroni:2004}.

Section~\ref{sect:bounds} describes the effect that upper and lower
slowly decaying bounds on $b$ have on the autocovariance function.
The main result is that a nontrivial subexponential upper bound on
the rate of decay of the autocovariance function is equivalent to a
nontrivial subexponential upper bound on the decay rate of the
kernel $b$. However, a numerical example demonstrates that a
corresponding lower bound on the autocovariance function does not
necessarily come from a corresponding lower bound on $b$, so one
cannot readily characterise necessary and sufficient conditions for
lower bounds on the memory of \eqref{eq:1a}.
Section~\ref{sect:bounds} also gives necessary and sufficient
conditions for exponential decay of the autocovariance function.
This last result complements the sufficient conditions of
\cite[Theorem~3.1]{pkrl:2000} while employing a different method of proof.

One of the chief differences in the analysis of this paper to that
of \cite{lgpkrl:2000, lgds:2002, pkrl:2000} is that rather than
analysing an explicit representation of the solution of
\eqref{eq:1a}, we primarily express the autocovariance function and
its associated resolvent as the solutions of Volterra equations and
then employ admissibility theory of linear Volterra operators to
study the asymptotic behaviour. Such admissibility theory has been
developed and used by Appleby, Gy{\H o}ri, Horv\'ath, Reynolds
\cite{appleby, jaigdr:2006, jaigdr:2011, islh} to determine rates of convergence to  the
equilibrium of linear Volterra summation equations. The proofs of
results stated in Sections~\ref{sect:sub} and~\ref{sect:bounds} are
confined to Section~\ref{sect:proofs}.

In this work, we have concentrated solely on the asymptotic
behaviour of stationary solutions of ARCH($\infty$) equations. It is
our belief that many of the asymptotic results presented here are
robust to mild departures from stationarity. However, an
investigation of this conjecture is deferred to a later work.
\section{Preliminaries}\label{sect:prelim}
Let $\Z$ be the set of integers, $\Z^{+}=\{n\in\Z:n\geq0\}$ and $\R$ the set of real numbers. If $d$ is a
positive integer, $\R^d$ is the space of $d$-dimensional column vectors with real components and
$\R^{d\times d}$ is the space of all $d \times d$ real matrices.
 We employ at various points the standard Landau order notation (cf e.g., \cite[Chapter~8.1]{se:1996}). Let $f$ and
$g$ be two functions defined on $\Z$ or $\R$. Then the notation
$f(n)\sim g(n)$ as $n\to \infty$ means that $\lim_{n\to \infty}
f(n)/g(n) = 1$. Sequences $u=\{u(n)\}_{n\geq 0}$ in $\R^d$ or
$U=\{U(n)\}_{n\geq 0}$ in $\R^{d\times d}$ are sometimes identified
with functions $u:\Z^{+}\to\R^d$ and $U:\Z^{+}\to\R^{d\times d}$. If
$\{U(n)\}_{n\geq0}$ and $\{V(n)\}_{n\geq 0}$ are sequences in
$\R^{d\times d}$, we define the \emph{convolution} of $\{(U\ast
V)(n)\}_{n\geq0}$ by
\[
    (U\ast V)(n)=\sum_{j=0}^n U(n-j)V(j), \quad n\geq 0.
\]
In this paper the \emph{$Z$-transform} of a sequence $U$ in $\R^{d\times d}$ is the function defined by
\[
    \tilde{U}( \lambda)= \sum_{j=0}^\infty U(j) \lambda^{j},
\]
provided $\lambda$ is a complex number for which the series
converges absolutely. A similar definition pertains for sequences
with values in other spaces. We remark that this definition of the
$Z$-transform differs from the more usual definition (see e.g.
\cite[Chapter~6.1]{se:1996}) in that $\lambda$ plays the role of $\lambda^{-1}$
and hence roots and poles of the $Z$-transform which were outside
the unit circle are now inside the unit circle, and vice versa.

For random variables $U$ and $V$ defined on the same probability
space, and which each have finite variance, we denote their means by
$\mathbb{E}[U]$ and $\mathbb{E}[V]$ and their variances by $\Var[U]$
and $\Var[V]$. Their covariance is denoted by $\Cov(U,V)$. A
stochastic process $X=\{X(k):k\in\mathbb{Z}\}$ is said to be
\textit{weakly stationary} if it has constant mean, $\EE[X(k)]\in\R$
for all $k\in\Z$, and there exists a function, called the
\textit{autocovariance function}, $\rho=\{\rho(k),k\in\Z\}$ such
that,
\begin{equation}\label{eq:cov4}
    \rho(k) = \Cov[X(n), X(n+k)], \quad \text{ for all } n,k\in\Z.
\end{equation}
Throughout this work the qualifiers weak and weakly are dropped, and
we refer to such processes as being \textit{stationary} or
possessing the property of \textit{stationarity}. The concept of
stationarity is that a structure is imposed upon the statistical
properties of the process which gives the process a time--invariance
property. The \textit{autocorrelation function} of $X$ is defined by
$\rho(k)/\Var[X(0)]$ for $k\in \mathbb{Z}$, where $\Var[X(0)]$ is
non--trivial.

It is of special interest in this work to establish the rate at
which $\rho(k)\to 0$ as $k\to\infty$ and in particular to
investigate whether the process $X$ possesses long memory. A number
of definitions of long memory exist in the literature: here we adopt
one of the commonest, saying that $X$ has \emph{long memory} if
 the autocovariance function is not summable i.e.,
\begin{equation}\label{eq:cov3}
    \sum_{k=0}^{\infty}|\rho(k)|=+\infty.
\end{equation}
The underpinning idea of long memory is that realisations far in the
past do not fade away quickly and so have a bearing upon the present
and future development of the process. The significance of long
memory as a measure of the efficiency of a financial market is
discussed in e.g. Cont~\cite{Cont}.

\section{Discussion of Existing Results on ARCH($\infty$) Processes}\label{sec:autocov}
Throughout this article we use the notation
\begin{equation*}
    \lambda_1 = \EE[\xi(0)],\quad \lambda_2 = \EE[\xi(0)^2],\quad B = \sum_{j=1}^{\infty}b(j), \quad \sigma^2=\Var[\xi(0)]=\lambda_2-\l^2.
\end{equation*}
It is assumed throughout that both the first moment of $\xi$ is
finite and non--zero, i.e. $0<\lambda_1<\infty$. A zero mean of
$\xi$ results in $X$ reducing to the trivial solution, i.e. $X(k)=0$
a.s. for all $k\in\mathbb{Z}$. Also $\sigma=0$ is equivalent to the
shocks $\xi$ being a.s. constant, and is therefore not of interest.
Equally, the case $a=0$ is not of interest, for it is known in this
case that $X(k)=0$ a.s. for all $k\in\mathbb{Z}$ is the only
stationary solution of \eqref{eq:1a}, see e.g.
\cite[Theorem~2.1]{lgpkrl:2000}.

Furthermore if $b(j)=0$ for all $j\geq1$ then this results in the
degenerate case of a constant conditional volatility of $X$ in
\eqref{eq:1a}, thereby defeating the initial motivation for studying
ARCH processes. In this case, $X$ degenerates to a constant multiple
of the i.i.d. non-negative ``shocks". We thus argue it is reasonable
to assume that there exists at least one value in the sequence $b$
which is positive. For this reason, we have as a standing hypothesis
throughout the paper that
\begin{equation} \label{eq.S0} \tag{$\text{S}_0$}
\lambda_1\in (0,\infty), \quad a>0, \quad \sigma\in (0,\infty),
\quad b\not\equiv 0.
\end{equation}

With the added assumption that
\begin{equation}\label{eq:con1}\tag{\text{S}$_1$}
    \lambda_1B<1,
\end{equation}
it is shown in \cite{lgpkrl:2000} that $\EE[X(k)]=a\l/(1-\l B)<+\infty$ for all $k\in\Z$.

A moving average representation of the solution of \eqref{eq:1a} is
derived in \cite{lgpkrl:2000}. We briefly outline the construction
of this representation and use it to develop a Volterra equation
satisfied by the coefficients of this representation. The results
later in this work concur with \cite[Theorem~2]{Zaffaroni:2004},
namely that these coefficients determine the rate of decay of the
autocovariance function.

Let $\psi(L) = 1 - \l\sum_{j=1}^{\infty}b(j)L^j$, where $L$ is the
\textit{lag} or \textit{backward shift} operator which operates on a
process $Y=\{Y(k):k\in\mathbb{Z}\}$ according to
$L\bigl(Y(k)\bigr)=Y(k-1)$. Define $\nu(k) :=X(k) - \l\varsigma(k)$:
then from \eqref{eq:1a} we have
\begin{equation*}
    \psi(L)X(k) = a\l + \nu(k).
\end{equation*}
A moving average representation for $X$ is then obtained by applying
the operator $\psi^{-1}(L)$ across this equation. The existence of
such an inverse operator (on the closed unit circle in the complex
plane) is given in \cite{lgpkrl:2000} and the references contained
therein. This existence is chiefly guaranteed by the summability of
$b$, a consequence of \eqref{eq:con1} which is assumed throughout
this work. We now state Lemma~4.1 of \cite{lgpkrl:2000}, which is
also~\cite[Problem 8, Chapter 18]{rudin}.
\begin{lemma}\label{lm:inver}
    Suppose $\sum_{j=0}^{\infty}|\psi_j|<\infty$, $\psi(\lambda):=\sum_{j=0}^{\infty}\psi_j \lambda^j$,
    and $|\psi(\lambda)|>0$ for $|\lambda|\leq~1$. Then there exists a sequence $z=\{z(j):j\in\Z^+\}$
    such that $D(\lambda):=1/\psi(\lambda)=\sum_{j=0}^{\infty}z(j)\lambda^j$
    is well defined for all $|\lambda|\leq1$. Furthermore, $\sum_{j=0}^{\infty}|z(j)|<+\infty$.
\end{lemma}
 We state the theorem guaranteeing a moving average representation from \cite[Theorem~4.1]{lgpkrl:2000}.
\begin{theorem}\label{thm:mov}
    If condition \eqref{eq:con1} holds, then there is a solution $X$ of \eqref{eq:1a} which admits the representation
\begin{equation*}
    X(k) = \mathbb{E}[X(k)] + \sum_{j=0}^{\infty}z(j)\nu(k-j)
\end{equation*}
    where $\sum_{j=0}^{\infty}|z(j)|<\infty$ and the process $\nu$ satisfies
    $\mathbb{E}[\nu(k)|\mathcal{F}(k-1)]=0$ for each $k$,
    where $(\mathcal{F}(k))_{k\in\mathbb{Z}}$ is the natural
    filtration generated by $\xi$.
\end{theorem}
Moreover, in~\cite{lgpkrl:2000} it is shown that with the additional
assumption
\begin{equation}\label{eq:con2} 
    \lambda_{2}^{\frac{1}{2}}\sum_{j=1}^{\infty}b(j)<1,
\end{equation}
then \eqref{eq:1a} has a unique \emph{weakly} stationary solution,
and hence $\mathbb{E}[\nu(k)^2]<+\infty$.

In both \cite{lgds:2002} and \cite{Zaffaroni:2004} necessary and
sufficient conditions are derived for the existence of a weakly
stationary solution of \eqref{eq:1a}. For completeness we state next
a slightly reformulated variant of part of \cite[Theorem
3.1]{lgds:2002}, omitting those parts that are not relevant to our
investigation.
\begin{theorem}\label{thm:lgds3}
The following are equivalent
\begin{itemize}
  \item[(a)] \eqref{eq:con1} holds and
    \begin{equation}\label{eq:nsstat}\tag{\text{S}$_2$}
        \Omega:=\frac{\sigma}{\l}\left(\sum_{j=1}^{\infty}z(j)^2 \right)^{1/2} <1
    \end{equation}
    where $z$ is (well) defined by
    \[
 \frac{1}{1-\lambda_1\sum_{j=1}^{\infty}b(j)\lambda^j}=\sum_{j=0}^{\infty}z(j)\lambda^j,
 \quad |\lambda|\leq 1;
 \]
\item[(b)]
   A weakly stationary solution $X$ of \eqref{eq:1a} exists.
    \end{itemize}
    Both imply that there exists a unique, ergodic solution of
\eqref{eq:1a} which may be written as a convergent orthogonal
Volterra series.     Moreover, $\Cov[X(0),X(k)]\geq0$ and
    \begin{equation}\label{eq:covres}
       \Cov[X(0),X(k)]=\left(\frac{a\sigma}{1-\l B}\right)^2\frac{1}{1-\Omega^2} \, \chi_z(k), \quad \text{for} \quad k\in\Z,
    \end{equation}
    where
\begin{equation} \label{def.chiz}
\chi_z(k)=\sum_{j=0}^{\infty}z(j)z(j+|k|).
\end{equation}
\end{theorem}
While the explicit representation of $X$ as a convergent orthogonal
Volterra series is a key component in the proof of Theorem~\ref{thm:lgds3},
in order to keep this article concise we do not
state this explicit form in the above as it does not form part of
our analysis. We further comment that, as observed in
\cite{lgds:2002}, the condition \eqref{eq:nsstat} is weaker than
\eqref{eq:con2}, which is imposed in \cite{lgpkrl:2000}. Under
\eqref{eq:nsstat}, $X$ is weakly stationary and the autocovariance
function is a multiple of
 $\chi_z$ and hence is absolutely summable, thus ruling out
long memory. Moreover as $b\geq0$ by hypothesis, this gives, via
\eqref{eq:deltadiff2}, that $z~\geq0$ and hence, under the condition
\eqref{eq:nsstat}, Theorem~\ref{thm:lgds3} gives
$\Cov[X(n),X(n+k)]~\geq~0$. This observation concurs with that of
\cite{lgpkrl:2000} for the non-negativity of the autocovariance
function under \eqref{eq:con2}.

Under the conditions of Theorem~\ref{thm:lgds3}, the moving average
representation of Theorem~\ref{thm:mov} and \eqref{eq:covres} imply
that
\[
    \mathbb{E}[\nu(0)^2]=\left(\frac{a\sigma}{1-\l
    B}\right)^2\frac{1}{1-\Omega^2},
\]
and also that
\begin{equation}\label{eq:archvar}
\Var[X(0)]=\left(\frac{a\sigma}{1-\l
    B}\right)^2\frac{1}{1-\Omega^2}\sum_{j=0}^\infty z(j)^2
    =\left(\frac{a\sigma}{1-\l
    B}\right)^2\frac{1+ \lambda_1^2\Omega^2/\sigma^2}{1-\Omega^2}.
\end{equation}
The first result of this paper is the calculation of a Yule-Walker
style of representation for the autocovariance of \eqref{eq:1a}.
\begin{proposition}\label{prop:1}
Let \eqref{eq:con1} and \eqref{eq:nsstat} hold. Then $\rho$, as defined by \eqref{eq:cov4}, obeys
\begin {equation}\label{eq:cov}
        \rho(k) =
        \begin{cases}
            \l\sum_{j=-\infty}^{k-1}b(k-j)\rho(j), & \text{ if } k\in\{1,2,3,...\},\\
            \rho(0), & \text{ if } k=0,\\
            \rho(-k), & \text{ if } k\in\{-1,-2,-3,...\},
        \end{cases}
\end{equation}
where $\rho(0)$ is given by \eqref{eq:archvar}.
\end{proposition}
The proof of Proposition~\ref{prop:1}, in common with many of the main results of the paper, is postponed to the end.

Proposition~\ref{prop:1} shows that the autocovariance obeys a
Volterra summation equation with infinite delay. Since the chief
focus of this paper is to describe the asymptotic behaviour of
$\rho$, it is interesting to draw a distinction between the
potential asymptotic behaviour of $\rho$ and the asymptotic
behaviour of the autocovariance function of an equation with a
\emph{finite} number of lags. To this end consider an ARCH($q$)
rather than an ARCH($\infty$) process. Then the resulting
autocorrelation function, as described by e.g.,
Taylor~\cite[pp.77,95]{st:1986}, corresponds exactly to the
autocorrelation function of the AR($q$) process
\[
    W(k) = \sum_{j=1}^{q}\l b(j)W(k-j) + e(k), \quad k\in\mathbb{Z},
\]
where $e=\{e(k)\}_{k\in\Z}$ is an uncorrelated sequence of random variables with finite constant variance.
 Hence \eqref{eq:cov} reduces to the Yule--Walker equations:
\begin{align*}
    \rho(k) = \l\sum_{j=1}^{q}b(j)\rho(k-j), \quad k\in\{1,2,...\}.
\end{align*}
Thus, the autocovariance function satisfies a $q^{th}$--order linear
difference equation with constant coefficients. It is well--known
that if the ARCH process is to be weakly stationary, all solutions of
an auxiliary polynomial equation must lie inside the unit disc in
$\mathbb{C}$, and that this condition also forces the autocovariance
function to decay geometrically. Hence, for a finite history
equation with a stationary solution, the autocovariance function
must decay geometrically: polynomial decay is impossible.

Thus, the study of the autocovariance function of AR or ARCH models is bound--up with that of difference equations. It is then natural to ask what the asymptotic features of the solutions of unbounded equations of the form
\[
    y(k) = \sum_{j=0}^{k-1}u(k-i)y(i), \quad k\geq1,
\]
are for some $u:\Z\to\R$ and initial condition $y(0)$ and whether
such an equation could be regarded as an underlying equation for the
autocovariance function of some stationary times series. To the
former question: it is well known that the dynamics of this equation
allow both exponential and slower--than--exponential decay (see
e.g.,~\cite{mura:2009} for convergence rates in weighted $l^1$
spaces, \cite{jaigdr:2006} for exact rates in $l^\infty$ spaces, and
\cite{ElMur:1996a} for the characterisation of exponential decay).
As to the latter: while for a stationary time series this is an open question nevertheless for a non-stationary times series such an equation could describe a family of autocovariances indexed by an initial starting time $m\in\Z$ i.e. $k\mapsto \Cov[X(m),X(k)]=y_{m}(k)$.

The distinction between this work and \cite{lgpkrl:2000, lgds:2002,
pkrl:2000,Zaffaroni:2004} is that we exploit the fact that $z$ from
Lemma~\ref{lm:inver} and Theorem~\ref{thm:lgds3} may be written as
the solution of a Volterra summation equation.
\begin{lemma}\label{prop:ztran}
    Suppose, for any $R>0$, $\lambda_1\sum_{j=1}^{\infty}b(j)R^{j}<+\infty$ and
     $\psi(\lambda)=1-\lambda_1\sum_{j=1}^{\infty}b(j)\lambda^j$ for $|\lambda|\leq R$.
    Then the following are equivalent:
    \begin{enumerate}
        \item $D(\lambda):=1/\psi(\lambda)=\sum_{j=0}^{\infty}z(j)\lambda^j$ is well defined for $|\lambda|\leq R$, $\sum_{j=0}^{\infty}z(j) R^{j}<\infty$
     and
\begin{equation}\label{eq:deltadiff2}
    z(n) =  \l\sum_{j=0}^{n-1}b(n-j)z(j),    \quad n=1,2,...; \quad
    z(0)=1;
\end{equation}
        \item $\l\sum_{j=1}^{\infty}b(j)R^{j}<1$.
    \end{enumerate}
\end{lemma}
\begin{remark}
    We remark that in the case $R=1$ much of the above lemma is covered in Lemma~\ref{lm:inver}. We note however that in Lemma~\ref{prop:ztran} the necessity of the condition $\l\sum_{j=1}^{\infty}b(j)R^{j}<1$ for the summability of $z$ is drawn out.
\end{remark}
\begin{remark}
    It is elementary, using \eqref{eq:deltadiff2}, to show that \eqref{eq:covres} is a solution of \eqref{eq:cov}.
\end{remark}
We observe that $z$ may be thought of as a resolvent for
\eqref{eq:cov} where the summation term is broken into a sum up to
time $k-1$ and the remainder of the sum thought of as a perturbation
term, i.e.
\begin{equation}\label{eq:cov2}
    \rho(k)= \l\sum_{j=0}^{k-1}b(k-j)\rho(j) + f(k-1), \quad k\geq1,
\end{equation}
where $f(k)=\l\sum_{j=1}^{\infty}b(k+j+1)\rho(-j)$ and hence one has the variation of parameters formula
\begin{equation}\label{eq:varpar}
    \rho(k) = z(k)\rho(0) + \sum_{j=0}^{k-1}z(k-j-1)f(j), \quad k\geq1,
\end{equation}
(see e.g.,~\cite{se:1996}). We demonstrate the usefulness of this
formulation of the autocovariance function in the proof of
Theorem~\ref{thm:22}. As this paper primarily uses properties of
Volterra equations to derive its results, it is perhaps more
intuitive to regard $z$ as the solution of an associated resolvent
equation rather than the coefficients of a power series or moving
average representation as in \cite{lgpkrl:2000, lgds:2002,
Zaffaroni:2004}.
\begin{remark} \label{rem:oldstatcond}
Using \eqref{eq:deltadiff2} and \eqref{eq:con2}, we can show that
\eqref{eq:nsstat} holds. Recalling that \eqref{eq:con2} implies
\eqref{eq:con1}, we can thus independently verify  the sufficiency
of \eqref{eq:con2} for the weak stationarity of the solution of
\eqref{eq:1a} as shown in~\cite[Theorem~2.1]{lgpkrl:2000}.
\end{remark}
\begin{proof}[Proof of Remark~\ref{rem:oldstatcond}]
Applying the Cauchy--Schwartz inequality to the righthand side of \eqref{eq:deltadiff2} yields
    \begin{align*}
        z(n)^2 \leq \l^2 B\sum_{j=0}^{n-1}b(n-j)z(j)^2, \quad n\geq1.
    \end{align*}
    By summing both sides of this equation, and using the fact that  \eqref{eq:con2} implies that $z^2$ is summable,
    we obtain
    \[
      1+ \sum_{n=1}^\infty z(n)^2 \leq 1+\l^2 B\sum_{n=1}^\infty \sum_{j=0}^{n-1}b(n-j)z(j)^2
      = 1+\l^2 B^2 \sum_{j=0}^\infty z(j)^2.
    \]
    Since $z(0)=1$, we obtain $\sum_{j=1}^\infty z^2(j)\leq 1/(1-\lambda_1^2B^2)-1$. Using this bound and  \eqref{eq:con2} leads to
    \eqref{eq:nsstat}.
\end{proof}

\begin{remark} \label{rem:newstatcond}
We can use the fact that $z$ satisfies \eqref{eq:deltadiff2} to
obtain a condition on $b$ which implies the stationarity of $X$ and
which is sometimes weaker than the condition \eqref{eq:con2}. More precisely, we show that
\begin{equation} \label{eq:con3}
\lambda_2<\lambda_1^2+\frac{(1-\lambda_1 B)^2}{\sum_{j=1}^\infty b(j)^2}
\end{equation}
implies \eqref{eq:nsstat}, and that \eqref{eq:con2} implies \eqref{eq:con3} if
\begin{equation} \label{eq:newcondbetter}
\lambda_1 B < \frac{1-\sum_{j=1}^\infty b(j)^2/B^2}{1+\sum_{j=1}^\infty b(j)^2/B^2}.
\end{equation}
\end{remark}
\begin{proof}[Proof of Remark~\ref{rem:newstatcond}]
We start by noticing that \eqref{eq:con1} implies $z$ is
summable, and by summing on both sides of \eqref{eq:deltadiff2} it
can readily be shown that $\sum_{j=0}^\infty z(j)=1/(1-\lambda_1
B)$. Since $b$ and $z$ are non--negative, we may apply the
Cauchy--Schwartz inequality to the right--hand side of
\eqref{eq:deltadiff2} to get
\begin{equation*}
    z(n)^2 \leq \l^2\sum_{j=0}^{n-1} z(j) \cdot \sum_{j=0}^{n-1}b(n-j)^2z(j),    \quad n\geq 1.
\end{equation*}
Since $z^2$ is summable, we get
\[
\sum_{n=1}^\infty z(n)^2\leq \l^2\sum_{j=0}^\infty z(j) \cdot
\sum_{n=1}^\infty \sum_{j=0}^{n-1}b(n-j)^2z(j)
=\l^2\frac{1}{(1-\lambda_1 B)^2} \sum_{j=1}^\infty b(j)^2.
\]
Therefore by this estimate and \eqref{eq:con3}, we have
\[
\frac{\sigma^2}{\lambda_1^2}\sum_{j=1}^\infty z(j)^2
\leq
\frac{\lambda_2-\lambda_1^2}{\lambda_1^2}\cdot \l^2\frac{1}{(1-\lambda_1 B)^2} \sum_{j=1}^\infty b(j)^2 <1,
\]
which is \eqref{eq:nsstat}. We notice that \eqref{eq:con2} can be written as $\lambda_2 B^2<1$, so \eqref{eq:con2}
is stronger than \eqref{eq:con3} if
\[
1<\lambda_1^2 B^2 +\frac{(1-\lambda_1 B)^2}{\sum_{j=1}^\infty b(j)^2/B^2}.
\]
which is equivalent to \eqref{eq:newcondbetter}, because $\l B<1$.
\end{proof}

\section{Exact Rates of Decay of the Autocovariance Function in the Class $\mathcal{W}(r)$}\label{sect:sub}
\subsection{Subexponential decay in linear Volterra summation equations}\label{admissVol}
In ascertaining rates of decay of Volterra equations we use
admissibility theory of Volterra operators, see
e.g.~\cite{jaigdr:2006}. \cite{jajd:pode} illustrates this facet of
admissibility theory for a discrete time Volterra equation whose
solution is an autocovariance function. We mention some pertinent
results of this theory. Consider the linear convolution equation
\begin{equation}\label{eq:scalarconvA}
    x(n+1)=f(n)+\sum_{i=0}^{n}F(n-i)x(i),\quad n\geq 0;\quad
    x(0)=x_0\in\R,
\end{equation}
where $f:\Z^+\to \R$  and $F:\Z^+\to \R$.  This
problem  has a unique solution $x:\Z^{+}\to\R$. In the case that
$x(n)\to 0$ as $n\to\infty$, our aim is to  describe the exact rate
of decay of $x$.
Our method is to introduce a suitable sequence $\gamma=\{\gamma(n)\}_{n\geq 0}$ which decays to zero and then to examine
the behaviour of
\begin{equation}\label{eq:ydef}
\omega(n)=x(n)/\gamma(n),
\end{equation}
and show that $\omega$ converges to a non-trivial limit. It then follows that $x(n)\to 0$ as $n\to\infty$ at exactly
the same rate as $\gamma(n)\to 0$.

We define a suitable class of real-valued weight functions, which
was studied in \cite{jaigdr:2006}.
\begin{definition}  \label{def:gensubexp}
Let $r>0$ be finite. A real-valued sequence
$\gamma=\{\gamma(n)\}_{n\geq0}$ is in $\mathcal{W}(r)$ if
$\gamma(n)>0$ for all $n\geq0$, and
\begin{gather}  \label{eq:p2}
\lim_{n\to\infty} \frac{\gamma(n-1)}{\gamma(n)}=\frac{1}{r},
\quad \sum_{i=0}^\infty  \gamma(i)r^{-i}<\infty,\\
\lim_{m\to\infty} \biggl(\limsup_{n\to\infty}\frac{1}{\gamma(n)}\sum_{i=m}^{n-m} \gamma(n-i)\gamma(i)\biggr)=0.
\label{eq:p1}\end{gather}
\end{definition}

Observe that if $r<1$ and $\gamma\in \mathcal{W}(r)$, then $\gamma$ decays; whereas if $r>1$, $\gamma$ diverges.
If $\gamma$ is in $\mathcal{W}(1)$, it is called a \emph{subexponential sequence},
one reason being that if $\gamma$ is in $\mathcal{W}(1)$, then
\begin{equation}\label{eq:subexpdiverge}
\lim_{n\to \infty}\gamma(n)\kappa^n=\infty\quad \text{for all $\kappa>1$.}
\end{equation}
The terminology is analogous with subexponential functions
and distributions.
Of course if $\gamma$ is in $\mathcal{W}(r)$ and $\delta(n)=r^{-n}\gamma(n)$, then $\delta$ is in
$\mathcal{W}(1)$.

Examples of sequences in $\mathcal{W}(r)$ include, but are not
limited to, $\gamma(n)=r^{n}n^{-\alpha}$ for $\alpha>1$;
$\gamma(n)=r^{n}n^{-\alpha}\exp(-n^{\beta})$ for
$\alpha\in\mathbb{R}$, $0<\beta<1$; and $\gamma(n)=r^{n} e^{-n/(\log
n)}$. The sequences defined by $\gamma(n) =r^{n}$, and
$\gamma(n)=r^{n}n^{-\alpha}$, $\alpha\leq 1$ are \emph{not} in
$\mathcal{W}(r)$.

We divide the results of this section into a discussion of
subexponential rates of decay ($r=1$) and a discussion of
$\mathcal{W}(r)$ rates of decay for $r<1$. While the proofs of both
of these sections are treated together, we choose to present the
results separately in order to emphasise the subexponential
behaviour in \eqref{eq:p2} which falls just short of long memory and
which is perhaps of greater interest in the context of time series.
The principal difference in the statement of these decay results is
that for sequences which are in $\mathcal{W}(1)$ we further require
that they are asymptotic to non--increasing sequences, whereas a
sequence in the class $\mathcal{W}(r)$, for $r<1$, is asymptotic to
a non-increasing sequence by the first part of \eqref{eq:p2}. Hence
we define a subclass $\mathcal{W}^{\downarrow}(r)$ of
$\mathcal{W}(r)$ for $r\in(0,1]$ by
\begin{align*}
\mathcal{W}^{\downarrow}(r)
    &:= \{g:\Z^+\to(0,\infty) :\quad \text{$g\in\mathcal{W}(r)$ and there exists $\gamma:\Z^+\to(0,\infty)$} \\
    &\qquad \text{ such that $\gamma(n+1)\leq\gamma(n)$ for all $n\in\Z^+$ and $g(n)\sim\gamma(n)$ as $n\to\infty$} \}.
\end{align*}
We note that $\mathcal{W}^{\downarrow}(r)=\mathcal{W}(r)$ for $r<1$.
This additional monotonicity is in practice quite a mild assumption given that we are interested in
determining a rate of \emph{decay} of $\rho$. We require it to simplify the asymptotic analysis of
certain infinite sums.

If $\gamma$ is a real sequence with $\gamma(n)>0$ for all $n\geq0$ and
$\{u(n)\}_{n\geq 0}$ is a  sequence in $\R^{d_1\times d_2}$ such that
$\lim_{n\to\infty} u(n)/\gamma(n)$ exists, then this limit is denoted by $L_\gamma u$.
This notation enables us to state succinctly \cite[Theorem~3.2]{jaigdr:2006}.
\begin{theorem}\label{thm:3b}
Suppose that there is
a $\gamma$ in $\mathcal{W}(r)$ such that $L_{\gamma}f$ and $L_{\gamma}F$ both exist, and that
\begin{equation}\label{eq:thm3a}
    \sum_{i=0}^{\infty}r^{-(i+1)}|F(i)|<1.
\end{equation}
Then the solution $x$ of \eqref{eq:scalarconvA} satisfies
\begin{equation}\label{eq:thm3b}
    L_{\gamma}x=\biggl(r- \sum_{i=0}^{\infty}r^{-i}F(i) \biggr)^{-1}[L_{\gamma}f+(L_{\gamma}F) \sum_{j=0}^{\infty}r^{-j}x(j)],
\end{equation}
where
\begin{equation}\label{eq:xtilde}
  \sum_{j=0}^{\infty}r^{-j}x(j)=\biggl(r- \sum_{k=0}^{\infty}r^{-k}F(k) \biggr)^{-1}[rx_0+\sum_{l=0}^{\infty}r^{-l}f(l)].
\end{equation}
\end{theorem}

\subsection{Necessary and sufficient conditions for subexponential decay.}\label{W(1)}
Our first main results show that subexponential decay in $b$ implies subexponential decay in $\rho$, and
moreover that $\rho$ decays at exactly the same rate as $b$.
\begin{theorem}\label{thm:4}
    Let \eqref{eq:nsstat} and $\l\sum_{j=1}^{\infty}b(j)<1$ hold.
 If $b\in\mathcal{W}^{\downarrow}(1)$ then $\rho\in\mathcal{W}^{\downarrow}(1)$.
          Moreover,
        \begin{equation}\label{eq:rholim1}
        L_{b}\rho = \frac{\l}{\bigl(1-\l B\bigr)}\sum_{j=-\infty}^{\infty}\rho(j)
        =\frac{\l\mathbb{E}[\nu(0)^2]}{\bigl(1-\l B\bigr)^3}.
    \end{equation}
\end{theorem}
The proof of Theorem~\ref{thm:4} is a consequence of
Theorems~\ref{thm:lgds3} and ~\ref{thm:3b}. This result is strongly
related to \cite[Theorem~2]{Zaffaroni:2004}, about which we comment
presently. The limit on the righthand side of \eqref{eq:rholim1} is
zero only when $a\sigma=0$, which is ruled out under the standing
assumptions \eqref{eq.S0} discussed at the beginning of
Section~\ref{sec:autocov}. The limit formulae \eqref{eq:rholim1}
highlights the inherent short memory of stationary solutions of
ARCH($\infty$) equations, because the infinite sum can be expressed
in terms of a finite quantity.

A simple corollary of this result is that if $b$ obeys $b(k)/k^{-\alpha}\to c>0$ as $k\to\infty$ for some
$\alpha>1$, and \eqref{eq:nsstat} and $\l\sum_{j=1}^{\infty}b(j)<1$ also hold, then $b\in \mathcal{W}^{\downarrow}(1)$, and we have
\[
\lim_{k\to\infty} \frac{\rho(k)}{k^{-\alpha}}=c'>0.
\]
We notice that this strengthens slightly results in \cite{lgpkrl:2000} and \cite{lgds:2002}, which give
upper and lower polynomial bounds on the rate of decay.

The necessity of subexponential decay in $b$ is captured by the
following result, which to the best of the authors' knowledge, is
not analogous to known results in the time series literature. It
shows, under an additional stability condition to that in
Theorem~\ref{thm:4}, that if $\rho$ is decaying subexponentially,
then $b$ must decay subexponentially, and at the same rate.
\begin{theorem}\label{thm:4a}
    Let \eqref{eq:nsstat} and $\l\sum_{j=1}^{\infty}b(j)<1/2$ hold. Then $b\in\mathcal{W}^{\downarrow}(1)$
    if and only if $\rho\in\mathcal{W}^{\downarrow}(1)$, and both statements imply
       \eqref{eq:rholim1}.
\end{theorem}
In the same spirit, we establish later in the paper a corresponding
pair of results for sequences in $\mathcal{W}(r)$, as well as
necessary and sufficient conditions for $\rho$ to be bounded above
by a subexponential sequence.

A novel feature of the proof of Theorem~\ref{thm:4a} is that we deal
with the advanced difference equation \eqref{eq:covres}, rather than
a Volterra equation. The proof of this partial converse is more
delicate than that of Theorem~\ref{thm:4} itself. It relies mainly on
showing that $\rho$ is asymptotic to $z$; once this is done, a known
result from the theory of Volterra difference equations ensures that
$z$ is asymptotic to $b$.

\subsection{Connections of Theorem~\ref{thm:4} with extant work}
Theorem~\ref{thm:4} (and Lemma~\ref{lm:basyd}) assert that, when $b$
is subexponential, then both $\rho$ (and $z$) inherit the rate of
decay of $b$. A result in almost exactly this direction is proven
in~\cite[Theorem 2]{Zaffaroni:2004}. There, it is claimed that if
\eqref{eq:con1} holds (which forces $b$ to be summable) and
\begin {equation} \label{eq:slower}
    \lim_{k\to\infty}\frac{b(k)}{\zeta^k} =\infty, \quad \text{ for any } 0<\zeta<1,
\end{equation}
then
\begin{align} \label{eq:dasp}
    z(k) &\sim C_1 b(k) \quad  \text{ and } \quad \chi_{z}(k) \sim C_2 b(k), \quad \text{ as } k\to\infty,
\end{align}
where $C_1,C_2\in(0,\infty)$ and $\chi_z$ is as defined in \eqref{def.chiz}.
The first asymptotic estimate appears as part of the proof
of~\cite[Theorem 2]{Zaffaroni:2004}, but the statement of the
theorem lists only the second estimate as its conclusion.

It should be noted that when $b\in \mathcal{W}(1)$, it obeys the first condition in \eqref{eq:p2} (with, by definition, $r=1$),
and therefore obeys \eqref{eq:subexpdiverge} which is equivalent to \eqref{eq:slower}. Therefore, at a first
glance, it would appear that Theorem~\ref{thm:4} proves the same
result as in~\cite[Theorem 2]{Zaffaroni:2004}, but requires stronger
hypotheses, as $\mathcal{W}(1)$ is merely a subclass of the summable
sequences obeying \eqref{eq:slower}.

Despite this, we now show that there exist sequences $b$ which obey
\eqref{eq:slower}, and which also satisfy the other conditions of
\cite[Theorem 2]{Zaffaroni:2004}, but for which the claimed
asymptotic behaviour for $z$ and $\chi_z$ in \eqref{eq:dasp} does
not hold. Notably, the sequences we consider are ruled out under the
stronger conditions of Theorem~\ref{thm:4} above. In essence, we
show that if $b$ does not obey the first condition in \eqref{eq:p2}
due to the presence of a 2-periodic component in its decay, then
this 2--periodic component is present in the rates of decay of $z$
and of $\chi_z$. Furthermore, this decay is ``out of phase'', in the
sense that neither $z$ nor $\chi_z$ are asymptotic to $b$, and
therefore violate \eqref{eq:dasp}.

The example we cite has been explored in detail in~\cite{jajd:pode}
(see Examples 4.2 and 4.5 and Remarks 4.4 and 4.6
in~\cite{jajd:pode}). However, to make our presentation
self--contained, we restate the main details of these results and
comments here and examine a specific numerical example. Scrutinising
the presentation in \cite{jajd:pode}, it can be seen that the
example can be generalised to cover any rate of decay in
$\mathcal{W}(1)$.
\begin{example}
Let $b(n)=a_1 n^{-2}$ for $n/2\in \mathbb{N}$ and $b(n)=a_0 n^{-2}$ for $n/2\not\in \mathbb{N}$ where $a_0=0.5$ and $a_1=0.25$.
Also, let $\{\xi(n)\}_{n\in \mathbb{N}}$ be a sequence of independent and identically distributed non--negative random variables
with mean $\lambda_1=1$. Note that
\[
\lim_{n\to\infty}\frac{b(2n+1)}{b(2n)} = 2, \quad \lim_{n\to\infty}\frac{b(2n+2)}{b(2n+1)} = \frac{1}{2},
\]
so that $b$ does not obey the first part of \eqref{eq:p2} for $r=1$
(or indeed any value of $r$), but does obey \eqref{eq:slower}. Since
\eqref{eq:con1} holds, \cite[Theorem 2]{Zaffaroni:2004} predicts
that there exist $C_1, C_2\in (0,\infty)$ such that
\[
\lim_{n\to\infty} \frac{z(n)}{b(n)}=C_1, \quad \lim_{n\to\infty} \frac{\chi_z(n)}{b(n)}=C_2,
\]
while Theorem~\ref{thm:4} does not apply.

However, by applying \cite[Theorem 3.2]{jajd:pode}, to this example and setting $\phi(n)=n^{-2}$ for $n\geq 1$ and $\phi(0)=2$,
explicit calculations in \cite[Example 4.2, 4.3]{jajd:pode} demonstrate that
    we have $b(2n+i+1)/\phi(2n)\to a_i>0$ for $i\in\{0,1\}$ and $\phi\in\mathcal{W}(1)$ with $a_0\not=a_1$, and
\begin{align*}
    d_0 := \lim_{n\to\infty}\frac{z(2n)}{\phi(2n)} = a_0T_0 + a_1T_1, \qquad
    d_1 := \lim_{n\to\infty}\frac{z(2n+1)}{\phi(2n)} = a_1T_0 + a_0T_1,
\end{align*}
where
    $T_0 = \Lambda(2S_0(1-S_1)),$   $T_1 = \Lambda(S_0^2 + (1-S_1)^2)$,
    $\Lambda = \bigl((1-S_1)^2-S_0^2\bigr)^{-2}$ and $S_i = \l\sum_{j=0}^{\infty}b(2j+i+1)$.
In this specific example, it can be shown that
\begin{align*} 
    S_0 = a_0\sum_{j=0}^{\infty}\frac{1}{(2j+1)^2} = \frac{\pi^2}{16}, \quad
    S_1 = a_1\sum_{j=0}^{\infty}\frac{1}{2^{2}(j+1)^{2}} = \frac{\pi^2}{96}.
\end{align*}
and noting that $S_0+S_1<1$, one can evaluate $\Lambda, T_0$ and $T_1$ respectively and hence $d_0$ and $d_1$. Indeed
$\Lambda = 5.55073...$, $T_0 = 6.14391...$ and $T_1 = 6.58015...$, which gives $d_0 = 4.71699...$ and $d_1 = 4.82605....$
Therefore
\begin{gather*}
\lim_{n\to\infty} \frac{z(2n)}{\phi(2n)}=d_0, \quad \lim_{n\to\infty}  \frac{z(2n)}{b(2n)}=d_0/a_1=4d_0, \\
\lim_{n\to\infty} \frac{z(2n+1)}{\phi(2n+1)}=d_1, \quad \lim_{n\to\infty}  \frac{z(2n+1)}{b(2n+1)}=d_1/a_0=2d_1,
\end{gather*}
and $4d_0\neq 2d_1$. Hence the claim of the first statement of \eqref{eq:dasp} does not hold.

Consulting \cite[Example 4.5 and Remark 4.6]{jajd:pode} shows that, under the above conditions, we have
\begin{align*}
    \lim_{k\to\infty}\frac{\chi_{z}(2k)}{\phi(2k)} =
     a_0\tau_0+a_1\tau_1, \quad
    \lim_{k\to\infty}\frac{\chi_{z}(2k+1)}{\phi(2k)} =
     a_0\tau_1+a_1\tau_0,
\end{align*}
where
\[
    \tau_0 = T_0\sum_{j=0}^{\infty}z(2j) + T_1\sum_{j=0}^{\infty}z(2j+1), \quad
    \tau_1 = T_1\sum_{j=0}^{\infty}z(2j) + T_0\sum_{j=0}^{\infty}z(2j+1).
\]
Thus for $\chi_{z}\sim b$ we need
$\lim_{k\to\infty} \chi_{z}(2k)/b(2k) =\lim_{k\to\infty} \chi_{z}(2k+1)/b(2k+1)$,
which is equivalent to $\tau_0(a_0-a_1)(a_0+a_1)/(a_0a_1)=0$,
which can only occur if $\tau_0=0$. To rule this out, note that summing over \eqref{eq:deltadiff2}
for both $z(2n)$ and $z(2n+1)$ gives
\[
    \sum_{j=0}^{\infty}z(2j)= \frac{(1-S_1)}{(1-S_1)^2-S_0^2}, \quad \sum_{j=0}^{\infty}z(2j+1)=\frac{S_0}{(1-S_1)^2-S_0^2},
\]
Filling in the values of $S_0$, $S_1$, $T_0$ and $T_1$ enables us to
compute $\tau_0=22.5498...\neq 0$. Hence the limits are unequal. In
fact, we have
\begin{align*}
    \lim_{k\to\infty}\frac{\chi_{z}(2k)}{b(2k)} = \frac{a_0}{a_1}\tau_0+\tau_1 = 67.9375\ldots,
    \lim_{k\to\infty}\frac{\chi_{z}(2k+1)}{b(2k+1)} = \frac{a_1}{a_0}\tau_0+\tau_1 = 34.1128\ldots
\end{align*}
This contradicts the second statement in \eqref{eq:dasp}.
\end{example}

\subsection{Necessary and sufficient conditions for $\mathcal{W}(r)$ decay.}\label{W(r)}
If it is observed that the autocovariances of the ARCH($\infty$)
equations decay in a manner consistent with the class
$\mathcal{W}(r)$ for $r\in(0,1)$, then this can only occur if the memory
of the process, $b$, decays likewise.
\begin{theorem}\label{thm:6}
    Fix $r\in(0,1)$. Let \eqref{eq:nsstat} and $\l\sum_{j=1}^{\infty}b(j)r^{-j}<1$ hold.
If $b\in\mathcal{W}(r)$ then $\rho\in\mathcal{W}(r)$.
     Moreover,
    \begin{equation} \label{eq:rholimr}
         \lim_{n\to\infty}\frac{\rho(n)}{b(n)}
        =\frac{\mathbb{E}[\nu(0)^2]}{(1-\l\sum_{j=0}^{\infty}b(j)r^j)}\cdot\frac{\l}{(1-\l\sum_{j=1}^{\infty}b(j)r^{-j})^2}.
    \end{equation}
\end{theorem}
A converse corresponding to Theorem~\ref{thm:4a} may also be stated.
\begin{theorem}\label{thm:6a}
    Fix $r\in(0,1)$. Let \eqref{eq:nsstat} and $\l\sum_{j=1}^{\infty}b(j)r^{-j}<1/2$ hold. Then
 $b\in\mathcal{W}(r)$ if and only if $\rho\in\mathcal{W}(r)$ and both imply \eqref{eq:rholimr}.
\end{theorem}

We remark that the rate of decay exhibited by a function in the
weight class of functions $\mathcal{W}(r)$, for $r<1$, is faster
than a purely geometric rate of decay. Let $b\in \mathcal{W}(r)$,
for $r<1$, and suppose that the conditions of Theorem~\ref{thm:6}
hold. Consider the open disc $D=\{\lambda\in \mathbb{C}:|\lambda|<
1/r\}$ of radius $1/r$ in the complex plane. Then the $Z$-transform
of $b$ is defined on $D$ and on the boundary of $D$, $\partial
D=\{\lambda\in \mathbb{C}:|\lambda|= 1/r\}$. Thus $\psi$, of Lemma~\ref{prop:ztran},
is well defined on $\bar{D}=D\cup \partial D$. However, by the conditions of
Theorem~\ref{thm:6}, $\psi$ has no zeroes in $\bar{D}$. Moreover,
because $b$ is in $\mathcal{W}(r)$, and $b(j)\geq 0$, we have
$\sum_{j=1}^\infty b(j)(1/r+\epsilon)^j=+\infty$ for every
$\epsilon>0$, and therefore neither the $Z$--transform of $b$, nor
$\psi$, are defined for real $\lambda>1/r$. Therefore the
characteristic equation $\psi(\lambda)=0$ excludes the possibility that there are
geometrically bounded solutions of $z$ at any rate $(1/|\lambda|)^n$
for $|\lambda|\leq 1/r$. On the other hand, Theorem~\ref{thm:6}
ensures that $z$ decays at the rate $r^n$ times a subexponential
sequence.

$\psi$ and the $Z$-transform of $b$ may be well defined in other
regions of the complex plane in the complement of $\bar{D}$, and
indeed $\psi$ may have zeroes in these other regions. Irrespective
of these potential zeroes, it is the $\mathcal{W}(r)$ rate of decay
of $b$ which determines the asymptotic behaviour of the resolvent
$z$ (i.e., the $\mathcal{W}(r)$ rate of decay dominates the
geometrically decaying solutions associated with the zeroes of
$\psi$). This analysis is consistent with Theorem~\ref{thm:22} which
describes a geometric decay. However, in light of the above
comments, it is apparent that this geometric decay rate need not be
given in terms of the roots of the characteristic equation.
\section{Bounds on the Decay Rate of the Autocovariance Function}\label{sect:bounds}
In this section we show that if there are decaying bounds imposed upon the kernel of
\eqref{eq:cov} then this forces the autocovariance function to also
be bounded with the same bounding decay rates. While the thrust of Section~\ref{sect:sub} was that
specific rates of decay of the kernel imply those same rates of
decay arising in the autocovariance function, we present an explicit
example where a bound in the rate of decay present in the autocovariance function
does not arise from the same rate of decay in the kernel.

Many of the results of this section hinge on the positivity of either $b$ or $\rho$ rather than merely on non--negativity.
Following on from the standing assumptions \eqref{eq.S0} at the
start of Section~\ref{sec:autocov}, we may assume that $b$ has at
least one positive component. Therefore, we are free to assume that
\begin{equation} \label{A1}
\text{
There exists a minimal $1\leq j^* <\infty$ such that $b(j^*) > 0$}. \tag{\text{A}$_1$}
\end{equation}
Then assuming \eqref{A1},
\[
    z(j^*)=\l\sum_{l=0}^{j^*-1}b(j^*-l)z(l)\geq\l b(j^*)>0
\]
and
\[
    \rho(j^*) = \mathbb{E}[\nu(0)^2]\sum_{l=0}^{\infty}z(l)z(l+j^*) \geq \mathbb{E}[\nu(0)^2]z(j^*)>0.
\]
 By \eqref{eq:cov}, for $k\geq0$ we see that
\begin{equation*}
    \rho(k+1) = \lambda_1\sum_{l=-\infty}^{k}b(k+1-l)\rho(l)
    \geq \lambda_1 b(k+1+j^*)\rho(-j^*), 
    \end{equation*}
    so
    \begin{align}
    \rho(k+1)\geq \lambda_1 b(k+1+j^*)\rho(j^*). \label{eq:nont}
\end{align}
Similarly, for all $k>j^*$,
    $z(k)\geq \lambda_1 b(k-j^*)z(j^*)$.

\begin{theorem}\label{thm:bounds}
Let $r\in(0,1]$ and suppose that $\l\sum_{j=1}^{\infty}b(j)r^{-j}<1$ and \eqref{eq:nsstat}
hold. Let $\gamma\in\mathcal{W}^\downarrow(r)$ be such that $b(n)\leq \gamma(n)$ for all $n\geq 0$.
 Then
    \begin{equation}\label{eq:ineq}
    \text{There exists $C_2\in (0,\infty)$ such that }
        \rho(n)\leq C_2\gamma(n), \quad \text{ for all $n\geq0$}.
    \end{equation}
\end{theorem}
\begin{remark}
    It is to be observed that Theorem~\ref{thm:bounds} is concerned in part with bounds in the class
 of non--increasing functions in $\mathcal{W}(1)$, which is a wider
class than the class of summable hyperbolically decaying
functions examined in \cite[Proposition~3.2]{lgpkrl:2000} and \cite[Corollary~3.2]{lgds:2002}.
\end{remark}
We now show that the conditions of Theorem~\ref{thm:bounds} are sharp if we are to observe an upper bound on $\rho$ in
$\mathcal{W}^\downarrow(r)$. Then we mention a
result concerning lower bounds on the autocovariance function.
\begin{theorem}\label{thm:conub}
Suppose that \eqref{eq:con1} and \eqref{eq:nsstat} hold and suppose that $\gamma\in\mathcal{W}^\downarrow(r)$ for $r\in (0,1]$.
Then the following are equivalent
\begin{itemize}
\item[(a)] $\l\sum^{\infty}_{j=1}b(j)r^{-j}<1$ and there exists $C_0\in (0,\infty)$ such that
\[
b(n)\leq C_0\g(n) \quad \text{ for all $n\geq1$};
\]
\item[(b)] There exists $C_2\in (0,\infty)$ such that
\[
\rho(n)\leq C_2\g(n) \quad \text{ for all $n\geq0$}.
\]
\end{itemize}
\end{theorem}
Theorem~\ref{thm:bounds} asserts that (a) implies (b). In the proof
that (b) implies (a) the resulting bound on $b$ is immediate from
\eqref{eq:nont}, while $\l\sum^{\infty}_{j=1}b(j)r^{-j}<1$ must
hold, as $z\leq C_1\g$, and so $\tilde{z}(r^{-1})<\infty$. Therefore
the proof of Theorem~\ref{thm:conub} is omitted.
\begin{theorem}\label{thm:lowb}
   Suppose that \eqref{eq:con1} and \eqref{eq:nsstat} hold and suppose that $\gamma\in\mathcal{W}^\downarrow(r)$ for $r\in (0,1]$.
   If there exists $C_0\in (0,\infty)$ such that $b(n)\geq C_0\g(n)$ for all $n\geq1$ then
   there exists $C_2\in (0,\infty)$ such that $\rho(n)\geq C_2\g(n)$ for all $n\geq0$.
\end{theorem}
The proof of Theorem~\ref{thm:lowb} is similarly omitted as it is
immediate from \eqref{eq:nont}. Combining the last two results gives
the main result of this section.
\begin{theorem}\label{thm:subsup}
Suppose that \eqref{eq:con1} and \eqref{eq:nsstat} hold and suppose that $\gamma\in\mathcal{W}^\downarrow(r)$ for $r\in (0,1]$.
Then the following are equivalent
\begin{itemize}
\item[(a)] $\l\sum^{\infty}_{j=1}b(j)r^{-j}<1$ and there exists $C_0^\ast\in (0,\infty)$ such that
\[
\limsup_{n\to\infty} \frac{b(n)}{\g(n)}=C_0^\ast;
\]
\item[(b)] There exists $C_2^\ast\in (0,\infty)$ such that
\[
\limsup_{n\to\infty} \frac{\rho(n)}{\g(n)}=C_2^\ast.
\]
\end{itemize}
\end{theorem}
\begin{remark}
    Theorem~\ref{thm:subsup} allows subsequences of $b$ to decay at rates faster than subexponentially, or indeed to be equal
    to zero. In this respect Theorem~\ref{thm:subsup} is different from the related result Theorem~\ref{thm:4}.
    Indeed the nature of the decay of $b$ may be quite erratic,
    yet providing that there is a subexponential decay which is an upper limiting bound for some subsequence of $b$
     then this limiting upper bound must be found in the autocovariance function and conversely.
\end{remark}
\begin{remark}
It is interesting to investigate what Theorem~\ref{thm:subsup} claims in the case when $r=1$. Suppose that
there is a stationary solution $X$ of \eqref{eq:1a}. Then Theorem~\ref{thm:lgds3} shows that conditions \eqref{eq:con1}
and \eqref{eq:nsstat} hold. If, from observation of the time series data, a subexponential sequence $\gamma$
is proposed for which $\limsup_{n\to\infty} \rho(n)/\gamma(n)\in (0,\infty)$, then Theorem~\ref{thm:subsup} shows that  $\limsup_{n\to\infty} b(n)/\gamma(n)\in (0,\infty)$.
\end{remark}
\begin{remark}
It is interesting to ask whether an analogue of
Theorem~\ref{thm:subsup} can be proven with the limit inferior in place
of the limit superior, for even though it is obvious from
\eqref{eq:nont} that $\liminf_{n\to\infty} b(n)/\gamma(n)>0$ implies
$\liminf_{n\to\infty} \rho(n)/\gamma(n)>0$, it is not so obvious
whether in general the converse holds. In Example~\ref{ex:sub0}
below, we demonstrate via a counterexample that this converse does
not hold in general. Therefore, it is also the case that the
converse of Theorem~\ref{thm:lowb} is not generally true.
\end{remark}
\begin{example}\label{ex:sub0} 
    Define the kernel $b$ so that it exhibits some periodicity:
    \[
    b(n) =
        \begin{cases}
            0, &\quad n/3\in\Z^+, \\
            n^{-2}, &\quad \text{otherwise}.
        \end{cases}
    \]
    Note that $\sum_{j=1}^\infty b(j)=4\pi^2/27$. Suppose that the sequence of shocks $\xi=\{\xi(n)\}_{n\in\Z}$ is such that
    $0<\l< 27/(4\pi^2)$, so that \eqref{eq:con1} holds.
    Following the techniques of \cite{jajd:pode} and the examples contained therein, we obtain
        \[
        \liminf_{n\to\infty}\frac{z(n)}{n^{-2}}=K\min\{d_0,d_1,d_2\}>0,
    \]
    where
    \[
        S_i = \l\sum_{n=0}^{\infty}b(3n+i+1), \quad i\in\{0,1,2\},
\]
        and
 \begin{align*}
        K &= \l/(1-S_0^3-3S_0S_1-S_1^3)^2, \\
        d_0 &= S_0^4+2S_1(1-S_0^3)+2S_0(1-S_1^3)+3(S_0^2+S_1^2)+S_1^4, \\
        d_1 &= 1+2S_0^3(1-S_1)+2S_1+2S_1^3+S_1^4+3S_0^2(1+S_1^2), \\
        d_2 &= 1+2S_1^3(1-S_0)+2S_0+2S_0^3+S_0^4+3S_1^2(1+S_0^2).
    \end{align*}
    Note that the denominator of $K$ is non--zero if $S_0>0$, $S_1>0$ and $S_0+S_1<1$.
    Similarly one may show that
    \[
        \liminf_{n\to\infty}\frac{\chi_{z}(n)}{n^{-2}}=\min\{c_0,c_1,c_2\}>0,
    \]
    where $\chi_z$ is defined by \eqref{def.chiz} and
    \begin{align*}
        c_0 &= d_0\sum_{j=0}^{\infty}z(3j)+d_1\sum_{j=0}^{\infty}z(3j+1) + d_2\sum_{j=0}^{\infty}z(3j+2), \\
        c_1 &= d_1\sum_{j=0}^{\infty}z(3j)+d_2\sum_{j=0}^{\infty}z(3j+1) + d_0\sum_{j=0}^{\infty}z(3j+2), \\
        c_2 &= d_2\sum_{j=0}^{\infty}z(3j)+d_0\sum_{j=0}^{\infty}z(3j+1) + d_1\sum_{j=0}^{\infty}z(3j+2).
    \end{align*}
    Noticing that $\sum_{j=1}^\infty b(j)^2=8\pi^4/729$, we see from Remarks~\ref{rem:oldstatcond} and \ref{rem:newstatcond} that if
    \[
\lambda_2 \frac{16\pi^4}{729} < 1+\max\left(0, \lambda_1^2 \frac{16\pi^4}{729}+2\left(1-\lambda_1 \frac{4\pi^2}{27}\right)^2 -1 \right),
    \]
    then  \eqref{eq:nsstat} also holds and one has $\liminf_{n\to\infty} \rho(n)/n^{-2}>0$.
    Therefore when the autocovariances of an ARCH($\infty$) process are observed to be
 bounded from below by a certain rate of decay, then it need \emph{not}
follow that this lower bounding rate of decay is present in $b$.

This example illustrates two further general points made earlier: first, in this example
$\limsup_{n\to\infty} b(n)/n^{-2}\in (0,\infty)$, and the above results confirm that
\[
\limsup_{n\to\infty} \rho(n)/n^{-2} = \mathbb{E}[\nu(0)^2]\max\{c_0,c_1,c_2\} \in (0,\infty),
\]
as claimed in Theorem~\ref{thm:subsup}.

Secondly, we notice from \eqref{eq:newcondbetter} that whenever $\lambda_1 < 9/(4\pi^2)$,
the condition \eqref{eq:con3}, which implies the stationarity of $X$, is weaker than condition \eqref{eq:con2}.
\end{example}

Using the subexponential bounds of Theorems~\ref{thm:conub} and \ref{thm:lowb}, we can weaken the hypothesis that $b$ is subexponential,
but still recover results on polynomial and ``superpolynomial'' decay of $\rho$. This is achieved at the expense of some
lost sharpness in characterising the asymptotic behaviour of $\rho$.
\begin{theorem}\label{thm:11}
Let \eqref{eq:con1} and \eqref{eq:nsstat} hold and $\beta\in\bigl\{(1,\infty)\cup\{\infty\}\bigr\}$.
        \begin{equation}\label{eq:logb}\tag{i}
        \text{If } \quad \lim_{n\to\infty}\frac{\log b(n)}{\log n} = -\beta \quad \text{ then }
        \quad \lim_{n\to\infty}\frac{\log \rho(n)}{\log n} = -\beta.
        \end{equation}
        \begin{equation}\label{eq:lnspb}\tag{ii}
            \limsup_{n\to\infty}\frac{\log b(n)}{\log n} = -\beta \quad  \text{ if and only if }
          \quad  \limsup_{n\to\infty}\frac{\log \rho(n)}{\log n} = -\beta.
        \end{equation}
\end{theorem}
Once again, we notice that the equivalence of the existence of a stationary solution of \eqref{eq:1a} and the conditions \eqref{eq:con1} and \eqref{eq:nsstat} means that the
``polynomial--like'' decay in the autocovariance function
exhibited in Theorem~\ref{thm:11} is possible if and only if similar ``polynomial--like'' decay is present in $b$.

Theorem~\ref{thm:11} can be used to determine the asymptotic behaviour for kernels $b$ which are not covered by previous results.
We can find examples of kernels $b$ for which
\[
\lim_{n\to\infty} \frac{\log b(n)}{\log n}=-\beta, \quad b\not\in \mathcal{W}(1)
\]
and also $b$ for which
\[
\limsup_{n\to\infty} \frac{\log b(n)}{\log n}=-\beta, \quad
\lim_{n\to\infty} \frac{\log b(n)}{\log n} \text{ does not exist},
\quad b\not\in \mathcal{W}(1).
\]
An example of the former is  $b(n)=(2+\cos(n\pi))n^{-\beta}$ or $b(n)= n^{-\beta}\log (n+2) (2+ \sin(n+2))$ while an example of the latter is $b(n)=n^{-\beta+\sin(n)-1}$ for $n\geq 1$. All these examples are not subexponential sequences as they fail to satisfy the first condition of \eqref{eq:p2}.

\begin{remark}
    Example~\ref{ex:sub0} shows that the first implication in Theorem~\ref{thm:11} cannot be reversed,
    as $\lim_{n\to\infty} \log\rho(n)/\log n=-2$, but $\lim_{n\to\infty} \log b(n)/\log n$ does not exist.
\end{remark}
\begin{remark}
    Theorem~\ref{thm:subsup} can be applied  when $b(n)=(2+(-1)^n) n^{-1} (\log (n+2))^{-2}$
    with e.g., $\gamma(n)=(n+2)^{-1} (\log (n+2))^{-2}\in\mathcal{W}(1)$, by following
    an adaption of the proof of \cite[Proposition~3.3]{jadwr:2002b}. However, Theorem~\ref{thm:11} does \emph{not} apply to this sequence.
\end{remark}

Despite the last remark, one may prefer Theorem~\ref{thm:11} over Theorem~\ref{thm:subsup} if the goal is to fit
real--world data to an $\text{ARCH}(\infty)$ model. In practice, one may not be able to
establish a subexponential sequence to which the data is ``close". In particular, it may only be possible to
identify the exponent of polynomial decay ($-\beta\in (-\infty,-1)$ in Theorem~\ref{thm:11}) in $b$ and not any lower order
component (for example logarithmic or other more slowly varying factors). Such difficulties might render impossible the detection of
the precise form of the subexponential sequence to which the kernel is close, particularly for sequences such as
$b(n)=n^{-\beta+\sin(n)-1}$.

In the final result, we show that exponential decay of $b$ is both necessary and sufficient for
exponential decay of $\rho$. Thus we recover a special case of \cite[Theorem~3.1]{pkrl:2000}, which concerns
exponential decay of the autocovariance function, while using a different method of proof.
\begin{theorem}\label{thm:22}
Let \eqref{eq:con1} and \eqref{eq:nsstat} hold. Then the following are equivalent:
\begin{itemize}
\item[(a)] There exist $\alpha_1\in (0,1)$, $C_1\in (0,\infty)$ such that  $b(k)\leq
C_1\alpha_1^{k}$ for all $k\in \mathbb{Z}^+$;
\item[(b)] There exist $\alpha_2\in (0,1)$, $C_2\in (0,\infty)$ such that  $\rho(k)\leq
C_2\alpha_2^{k}$ for all $k\in \mathbb{Z}^+$.
\end{itemize}
\end{theorem}
\section{Proofs}\label{sect:proofs}
Proposition~\ref{prop:1} necessitates that interchange of an
infinite summation and an expectation sign. This interchange is made
rigorous via standard application of the Monotone--Convergence
Theorem (cf. e.g.,~\cite[Theorem~5.3]{Williams}).
\begin{proof}[Proof of Proposition~\ref{prop:1}]
Firstly observe that the identity $\rho(k)=\rho(-k)$, for all
$k\in\Z$ holds for the autocovariance function.
Now, for $k>0$ we have
\begin{align*}
    \rho(-k) &=  \Cov[X(n),X(n-k)]= \Cov[a\xi(n)+\sum_{j=1}^{\infty}b(j)X(n-j)\xi(n),X(n-k)] \\
    &= a\,\Cov[\xi(n),X(n-k)] + \sum_{j=1}^{\infty}b(j)\Cov[X(n-j)\xi(n),X(n-k)] \\
    &= 0 + \l\sum_{j=1}^{\infty}b(j)\Cov[X(n-j),X(n-k)] = \l\sum_{j=1}^{\infty}b(j)\rho(k-j).
\end{align*}
The result follows due to the symmetry of the autocovariance
function.
\end{proof}
\begin{proof}[Proof of Lemma~\ref{prop:ztran}]
Firstly we note that $\lambda_1\sum_{j=1}^{\infty}b(j)R^j<+\infty$ ensures that $\psi(\lambda)$ is finite
in the region $|\lambda|\leq R$.

Suppose now that $\lambda_1\sum_{j=1}^{\infty}b(j)R^j<1$. Let $|\lambda|\leq R$.
Define $\Lambda:=\lambda/R$, so that $|\Lambda|\leq 1$. Also, define the sequence $\psi^*$ by $\psi^*_0=1$, $\psi^*_j=-\l b(j)R^j$
for $j\geq1$. Therefore $\sum_{j=0}^{\infty}|\psi^*_j|<+\infty$. Consequently, we may define
$\psi^*(\Lambda)=\sum_{j=0}^{\infty}\psi^*_j\Lambda^j$ for $|\Lambda|\leq 1$. Furthermore, for $|\Lambda|\leq 1$, we may use the non--negativity
of $b$ to get
    \begin{align*}
    |\psi^*(\Lambda)|&=
     |1-\l\sum_{j=1}^{\infty}b(j)R^j\Lambda^j|
    \geq 1-\l\sum_{j=1}^{\infty}b(j)R^j >0.
  \end{align*}
 Hence we may apply Lemma~\ref{lm:inver} to $\psi^\ast$, so that there exists a summable sequence $z^*=\{z^*(j):j\in\Z^+\}$ such that
 $1/\psi^*(\Lambda) = \sum_{j=0}^{\infty}z^*(j)\Lambda^j$ for $|\Lambda|\leq 1$. Therefore, for $|\lambda|\leq R$ we have
    \[
        \frac{1}{\psi(\lambda)}=\frac{1}{\psi^*(\Lambda)} = \frac{1}{\sum_{j=0}^{\infty}\psi^*_j\Lambda^j} = \sum_{j=0}^{\infty}z^*(j)\Lambda^j
        = \sum_{j=0}^{\infty}z^*(j)R^{-j}\lambda^j.
    \]
    Therefore
    \[
        \sum_{j=0}^{\infty}z^*(j)R^{-j}\lambda^j \sum_{k=0}^{\infty}\psi_k^* R^{-k}\lambda^k = 1, \quad |\lambda|\leq R.
    \]
    Note that when $R=1$, we have $z^*=z$ in the notation of Lemma~\ref{lm:inver}.
    Rearranging gives
    \begin{align*}
        \sum_{l=0}^{\infty}\sum_{j=0}^{l}\psi^*_{l-j}z^*(j)R^{-l}\lambda^l =1.
    \end{align*}
    Now comparing powers of $\lambda$ on both sides of this equality gives
    \begin{align}\label{eq:deltadiff3}
        \psi^*_0 z^*(0)=1 , \quad
        z^*(n)  = -\sum_{j=0}^{n-1}\psi^*_{n-j}z^*(j), \quad n\geq1.
    \end{align}
    Rearranging the second equation gives
    \[
        R^{-n}z^*(n) = \lambda_1\sum_{j=0}^{n-1}b(n-j)R^{-j}z^*(j), \quad n\geq 1.
    \]
    Observe that if $R=1$, $z^*$ satisfies \eqref{eq:deltadiff2}. Define $w(n)=R^{-n}z^*(n)$ for $n\geq0$. Then,
    by the uniqueness of the solution of \eqref{eq:deltadiff2}, it is seen that $w(n)=z(n)$, $n\geq0$ and so $z^*(n)=R^n z(n)$, $n\geq0$.
    Hence
        $1/ \psi(\lambda) = \sum_{j=0}^{\infty}z(j)\lambda^j ,\quad |\lambda|\leq R$
    and $\sum_{j=0}^{\infty}z(j)R^j<+\infty$.

Conversely, suppose that $z$ is defined by \eqref{eq:deltadiff2} and that $\sum_{j=0}^{\infty}z(j)R^{j}<+\infty$. Multiplying across \eqref{eq:deltadiff2}
by $R^{n}$ and summing gives
\[
\sum_{n=1}^\infty z(n)R^{n}=\lambda_1 \sum_{n=1}^\infty \sum_{j=0}^{n-1} b(n-j)R^{n-j}R^jz(j).
\]
Since the summand on the righthand side is non--negative, the order of summation may be exchanged to give
\[
    \sum_{n=0}^{\infty}z(n)R^{n} = 1+\l\sum_{j=1}^{\infty}b(j)R^{j}\sum_{n=0}^{\infty}z(n)R^{n}.
\]
Now, since $\sum_{n=0}^{\infty}z(n)R^{n} \in [1,\infty)$, it follows that $\l\sum_{j=1}^{\infty}b(j)R^{j}$ is finite, and moreover
the identity can be rearranged to give
\[
\l\sum_{j=1}^{\infty}b(j)R^{j}=\frac{\sum_{n=0}^{\infty}z(n)R^{n}-1}{\sum_{n=0}^{\infty}z(n)R^{n}}\in [0,1),
\]
as required.
\end{proof}
\subsection{Rates}
It is obvious from \eqref{eq:deltadiff2}
that if $\l\sum_{j=1}^{\infty}b(j)r^{-j}<1$ then
    \[
      \sum_{j=0}^{\infty}z(j)r^{-j}=\frac{1}{1-\l\sum_{j=1}^{\infty}b(j)r^{-j}}<+\infty
  \]
and trivially $\sum_{j=0}^{\infty}z(j)r^j<\infty$ and $\l\sum_{j=1}^{\infty}b(j)r^j<1$ for $r\in (0,1]$.

\begin{lemma}\label{lm:basyd}
        If $b\in\mathcal{W}(r)$ and $\l\sum_{j=1}^{\infty}b(j)r^{-j}<1$, then
        \[
            \lim_{n\to\infty}\frac{z(n)}{b(n)} = \frac{\l}{(1-\l\sum_{j=1}^{\infty}b(j)r^{-j})^2}.
        \]
    \end{lemma}
\begin{proof}[Proof of Lemma~\ref{lm:basyd}]
    Apply Theorem~\ref{thm:3b} to \eqref{eq:deltadiff2}.
\end{proof}

\begin{lemma}\label{lm:dasycov}
If $b\in\mathcal{W}^{\downarrow}(r)$ for $r\in(0,1]$, $\l\sum_{j=1}^{\infty}b(j)r^{-j}<1$, and $\chi_z$ is defined by
\eqref{def.chiz}, then
    \[
        \lim_{k\to\infty}\frac{\chi_{z}(k)}{z(k)} = \frac{1}{1-\l\sum_{j=1}^{\infty}b(j)r^j}.
    \]
\end{lemma}
\begin{proof}[Proof of Lemma~\ref{lm:dasycov}]
    Firstly, note that $\l\sum_{j=1}^{\infty}b(j)r^{-j}<1$ gives $\sum_{j=0}^{\infty}z(j)r^{-j}<+\infty$.
    Consider the case $r<1$. Then for any fixed $M\geq 2$ we have
    \[
        \left|\frac{\chi_{z}(n)}{z(n)}-\sum_{j=0}^{\infty}z(j)r^{j}\right|
        \leq \sum_{j=0}^{M-1}z(j)\left|\frac{z(n+j)}{z(n)}-r^j\right|
        + \sum_{j=M}^{\infty}z(j)\frac{z(n+j)}{z(n)} + \sum_{j=M}^{\infty}z(j)r^j.
    \]
    Let $\epsilon\in(0,1)$ be such that $r<r(1+\epsilon)<1<r^{-1}$. By Lemma~\ref{lm:basyd}
there is an $N(\epsilon)\in\Z^+$ such that
        $z(n+1)/z(n)<r(1+\epsilon)<1$ for all $n\geq N(\epsilon)$.
    Hence for $j\geq1$, $z(n+j)/z(n) < r^j(1+\epsilon)^j < r^{-j}$ for all $n\geq N(\epsilon)$. Thus for $n\geq N(\epsilon)$,
    \[
        \left|\frac{\chi_{z}(n)}{z(n)}-\sum_{j=0}^{\infty}z(j)r^{j}\right|
        \leq 2\sum_{j=M}^{\infty}z(j)r^{-j} + \sum_{j=0}^{M-1}z(j)\left|\frac{z(n+j)}{z(n)}-r^j\right|.
    \]
    Since $\lim_{n\to\infty} z(n+j)/z(n) = r^j$, we have
    \[
        \limsup_{n\to\infty}\left|\frac{\chi_{z}(n)}{z(n)}-\sum_{j=0}^{\infty}z(j)r^{j}\right|
            \leq 2\sum_{j=M}^{\infty}z(j)r^{j}.
    \]
    Finally, letting $M\to\infty$ gives the desired result for $r<1$.

    For the case $r=1$, we split the sums in the same manner as above. From Lemma~\ref{lm:basyd} we have
    that $z\in\mathcal{W}(1)$. Then we use the asymptotic monotonicity of $b$ to bound $z(n+j)/z(n)$.
We have for $n\geq N_1$, for some $N_1$ sufficiently large
    \[
        \lim_{n\to\infty}\frac{z(n)}{b(n)}= L\in(0,\infty),
        \quad \frac{b(n+j)}{b(n)}\leq\frac{b(n+j)}{\gamma(n+j)}\cdot\frac{\gamma(n)}{b(n)}\leq 2\cdot2 \text{ for all } j\geq1.
    \]
    where $\gamma$ is the non--increasing sequence which is asymptotic to $b$. Thus for
$n\geq N_1$
    \begin{align*}
        \frac{z(n+j)}{z(n)} = \frac{z(n+j)}{b(n+j)}\cdot\frac{b(n+j)}{b(n)}\cdot\frac{b(n)}{z(n)}
        \leq 2L\frac{b(n+j)}{b(n)}\frac{1}{L}2 \leq 2^4.
    \end{align*}
The result follows through as before.
\end{proof}
\begin{proof}[Proof of Theorems~\ref{thm:4} and~\ref{thm:6}]
    Theorem~\ref{thm:6} and the second limit in Theorem~\ref{thm:4} are an immediate consequence
 of Lemmas~\ref{lm:basyd} and~\ref{lm:dasycov} with
\eqref{eq:nsstat} being required to guarantee that
$\mathbb{E}[\nu(0)^2]$ is well defined and finite.

Turning to the first limit formula in Theorem~\ref{thm:4}, from Lemma~\ref{lm:dasycov} we have that
$\rho\in\mathcal{W}(1)$ and hence
$\sum_{j=0}^{\infty}\rho(j)<\infty$. From \eqref{eq:cov2} we
have
    \begin{equation}\label{eq:covcon}
        \rho(n+1) = \l\sum_{j=0}^{n}b(n-j+1)\rho(j)+f(n),
    \end{equation}
    where
$f(n) = \l\sum_{j=1}^{\infty}b(n+j+1)\rho(j)$.
    Letting $F(n)=\l b(n+1)$ we can then apply Theorem~\ref{thm:3b} to get a representation
 for $L_{b}\rho$, providing that $L_{\gamma}f$ and $L_{\gamma}F$ both
exist, and that $\sum_{j=0}^{\infty}F(j)<1$. We have the
last condition by assumption. To prove that $L_\gamma F$ exists, note that
    \[
        \lim_{n\to\infty}\frac{F(n)}{\g(n)} = \lim_{n\to\infty}\frac{\l b(n+1)}{\g(n)}
        = \lim_{n\to\infty}\frac{\l b(n+1)}{\g(n+1)}\frac{\g(n+1)}{\g(n)}= \l.
    \]
As to the existence of $L_\gamma f$, we fix $M\in\Z^+$, and make the estimate
    \begin{multline*}
        \left|\frac{f(n)}{\g(n)}-\l\sum_{j=1}^{\infty}\rho(j)\right|
        \leq \l\sum_{j=1}^{M}\left|\frac{b(n+j+1)}{\g(n)}-1\right|\rho(j)
        \\+ \l\sum_{j=M+1}^{\infty}\frac{b(n+1+j)}{\g(n)}\rho(j) +
        \l\sum_{j=M+1}^{\infty}\rho(j).
    \end{multline*}
    For the second term on the right hand side we have
    \[
        \frac{b(n+1+j)}{\g(n)} = \frac{b(n+1+j)}{\g(n+1+j)}\frac{\g(n+1+j)}{\g(n)}
        \leq 2,
    \]
    for all $n\geq N_0$ and some $N_0$ sufficiently large.
    Thus for $n\geq N_0$,
    \[
        \left|\frac{f(n)}{\g(n)}-\l\sum_{j=1}^{\infty}\rho(j)\right|
        \leq 3\l\sum_{j=M+1}^{\infty}\rho(j) + \l\sum_{j=1}^{M}\left|\frac{b(n+j+1)}{\g(n)}-1\right|\rho(j).
    \]
    Then
    \[
        \limsup_{n\to\infty}\left|\frac{f(n)}{\g(n)}-\l\sum_{j=1}^{\infty}\rho(j)\right|
        \leq 3\l\sum_{j=M+1}^{\infty}\rho(j).
    \]
    Letting $M\to\infty$ gives $L_\gamma f=\l\sum_{j=1}^{\infty}\rho(j)$.

    Thus we may apply Theorem~\ref{thm:3b}, which gives that $L_{b}\rho=L_{\g}\rho$ exists.
    Applying~\cite[Theorem~4.3 ]{jaigdr:2006} to \eqref{eq:covcon} gives
    \[
        L_{b}\rho = \frac{\l\sum_{j=0}^{\infty}\rho(j) + \l\sum_{j=1}^{\infty}\rho(j)}{1-\l\sum_{j=1}^{\infty}b(j)}.
    \]
    Using the symmetry of the autocovariance function, i.e., $\rho(n)=\rho(-n)$ for all $n\in\Z$, gives \eqref{eq:rholim1}
    as required.
\end{proof}
We provide a partial converse to Lemma~\ref{lm:basyd}, i.e., that $z\in\mathcal{W}(r)$
implies $b\in\mathcal{W}(r)$. To do so, we state without proof a variant of~\cite[Theorem~3.7]{jajd:pode}.
The proof of this consists of rewriting \eqref{eq:deltadiff2} so that the roles of $b$ and $z$ are
interchanged, and by then applying Theorem~\ref{thm:3b}.
\begin{lemma}\label{lm:bres}
    Let $z$ be the sequence which satisfies \eqref{eq:deltadiff2},
    $z\in\mathcal{W}(r)$ and further suppose that
    \begin{equation}\label{eq:bhalf}
        \l\sum_{j=1}^{\infty}b(j)r^{-j}<\frac{1}{2}.
    \end{equation}
    Then
    \[
        \lim_{n\to\infty}\frac{b(n)}{z(n)}= \frac{1}{\l\left(\sum_{j=0}^{\infty}z(j)r^{-j}\right)^2}.
    \]
\end{lemma}
\begin{remark}\label{rk:l1}
    If $r\in(0,1]$ and
        $\l\sum_{j=1}^{\infty}b(j)r^{-j}<\frac{1}{2},$
   then
        $\sum_{j=1}^{\infty}z(j)r^{-j}<1,$
    and hence
        $\l\sum_{j=1}^{\infty}b(j)r^{j}<\frac{1}{2}$
        \text{ and }
        $\sum_{j=1}^{\infty}z(j)r^{j}<1$.
\end{remark}
We now state some preparatory lemmata which lead to converses of Theorems~\ref{thm:4} and~\ref{thm:6}.
\begin{lemma}\label{lm:udlim}
    Let $z$ be the solution of \eqref{eq:deltadiff2} and let \eqref{eq:bhalf} hold with $r\in(0,1]$. Define
    the sequences $(U_m)_{m\geq 1}$ and $(L_m)_{m\geq 1}$ by
    \begin{align*}
        U_1=1, \quad U_{m+1}=1-\sum_{j=1}^{m}z(j)r^jL_m ,\quad  L_m= 1-\sum_{j=1}^{\infty}z(j)r^jU_m, \quad  m\in\Z^+/\{0\}.
    \end{align*}
    Then
    \[
        \lim_{m\to\infty}U_{m} = \lim_{m\to\infty}L_{m} = 1-\l\sum_{j=1}^{\infty}b(j)r^j.
    \]
\end{lemma}
 \begin{proof}[Proof of Lemma~\ref{lm:udlim}]
    The proof concentrates on verifying that $\lim_{m\to\infty}U_m$ exists. Once this limit is established
    it is easy to find $\lim_{m\to\infty}L_m$.
    We have $U_1=1$ and
    \[
        U_{m+1}=g(m) + a(m)U_m, \quad m\geq1,
    \]
    where $g(m) =1-\sum_{j=1}^{m}z(j)r^j$ and $a(m)=\sum_{j=1}^{\infty}z(j)r^j\sum_{l=1}^{m}z(l)r^l$. An explicit formula
    for $U$ is given in e.g.~\cite[Exercise~2.1.17]{se:1996} and is
    \begin{equation}\label{eq:Uex}
        U_{m+1}=\prod_{j=1}^{m}a(j)U_1 + \sum_{n=1}^{m}\bigl\{\prod_{j=n+1}^{m}a(j)\bigr\}g(n), \quad m\geq2,
    \end{equation}
    in which the usual convention $\prod_{j=m+1}^{m}a(j):=1$ applies.
    Also we note that $g(m)\to1-\sum_{j=1}^{\infty}z(j)r^j$ and $a(m)\to\bigl(\sum_{j=1}^{\infty}z(j)r^j\bigr)^2\in(0,1)$,
 as $m\to\infty$. Thus the first term on the right--hand side of
\eqref{eq:Uex} tends to zero as $m\to\infty$. Turning our attention
to the second term we have
    \[
        A_m:= \sum_{n=1}^{m}\frac{\prod_{j=1}^{m}a(j)}{\prod_{j=1}^{n}a(j)}g(n)
        = \frac{\sum_{n=1}^{m}\frac{1}{\prod_{j=1}^{n}a(j)}g(n)}{\frac{1}{\prod_{j=1}^{m}a(j)}}
        =\frac{\sum_{n=2}^{m}c(n)+c(1)}{\sum_{n=2}^{m}d(n)+\frac{1}{a(1)}},
    \]
    where
    \begin{align*}
        d(n):=\frac{1}{\prod_{j=1}^{n}a(j)}-\frac{1}{\prod_{j=1}^{n-1}a(j)}, \quad
        c(n):=\frac{1}{\prod_{j=1}^{n}a(j)}g(n).
    \end{align*}
    Thus $d(n)=\frac{1-a(n)}{\prod_{j=1}^{n}a(j)}$ and hence $c(n)\to\infty$ and $d(n)\to\infty$ as $n\to\infty$. Moreover,
    \[
        \frac{c(n)}{d(n)}=\frac{g(n)}{1-a(n)} = \frac{1-\sum_{j=1}^{n}z(j)r^j}{1-\sum_{j=1}^{\infty}z(j)r^j\sum_{l=1}^{n}z(l)r^l}
    \]
    and so
    \[
        \lim_{n\to\infty}\frac{c(n)}{d(n)}= \frac{1-\sum_{j=1}^{\infty}z(j)r^j}{1-\biggr(\sum_{j=1}^{\infty}z(j)r^j\biggl)^2}
        =\frac{1}{1+\sum_{j=1}^{\infty}z(j)r^j}.
    \]
    Applying Toeplitz's Lemma (cf., e.g.,~\cite[4.3.2 pp.390]{Shir}) now gives
    \[
        \lim_{m\to\infty}\frac{\sum_{n=2}^{m}c(n)}{\sum_{n=2}^{m}d(n)}=\frac{1}{1+\sum_{j=1}^{\infty}z(j)r^j}.
    \]
    Therefore
    \[
        \lim_{m\to\infty}U_m=\lim_{m\to\infty}A_m
        =\lim_{m\to\infty}\frac{\sum_{n=2}^{m}c(n)+c(1)}{\sum_{n=2}^{m}d(n)+\frac{1}{a(1)}}
        =\frac{1}{1+\sum_{j=1}^{\infty}z(j)r^j}.
    \]
    Finally, $z$ may be written in terms of $b$ using \eqref{eq:deltadiff2}.
\end{proof}
\begin{lemma}\label{lm:resb}
    Let \eqref{eq:nsstat} and \eqref{eq:bhalf} hold. If $\rho\in\mathcal{W}^{\downarrow}(r)$, for $r\in(0,1]$, then $z$ satisfies
    \begin{equation} \label{eq:boundzrhom}
        L_{m}\leq\mathbb{E}[\nu(0)^2]\liminf_{n\to\infty}\frac{z(n)}{\rho(n)}
        \leq\mathbb{E}[\nu(0)^2]\limsup_{n\to\infty}\frac{z(n)}{\rho(n)}\leq U_{m+1}, \quad m\geq1,
    \end{equation}
    where $U$ and $L$ are the sequences defined in Lemma~\ref{lm:udlim}.
\end{lemma}
\begin{proof}[ Proof of Lemma~\ref{lm:resb}]
    The upper and lower bounds on $z/\rho$ are established by an inductive proof. The bounds themselves are
    constructed recursively.  Define $P(n)=\rho(n)/\mathbb{E}[\nu(0)^2]$. We deal with the case when $r\in(0,1)$:
    the proof for $r=1$ is largely similar, but employs the asymptotic
monotonicity of $P$ to establish estimates for terms of the form
$P(n+j)/P(n)$.

    From \eqref{eq:covres} and using the non-negativity of $z$ and definition of $P$, we have
    \begin{equation}\label{eq:u1}
        P(n) = \sum_{j=0}^{\infty}z(j)z(n+j) = z(n)+ \sum_{j=1}^{\infty}z(j)z(n+j) \geq z(n).
    \end{equation}
    Thus $z(n)/P(n)\leq1$ and so $\limsup_{n\to\infty} z(n)/P(n)\leq1=U_1$.
    As $\lim_{n\to\infty}P(n+1)/P(n)=r$ we have for all $\epsilon>0$ fixed that there exists an $N_0(\epsilon)\in\Z^+$
    such that $P(n+j)/P(n)<r^j(1+\epsilon)^j<1<r^{-j}$ for all $n\geq N_0(\epsilon)$.
    Fix $M\in\Z^+$. Let $n\geq N_0$. Thus by \eqref{eq:u1}
    \begin{align*}
        \frac{1}{P(n)}\sum_{j=1}^{\infty}z(j)z(n+j)&\leq\frac{1}{P(n)}\sum_{j=1}^{\infty}z(j)P(n+j) \\
        &=\sum_{j=1}^{M}z(j)\frac{P(n+j)}{P(n)} + \sum_{j=M+1}^{\infty}z(j)\frac{P(n+j)}{P(n)} \\
        &\leq\sum_{j=1}^{M}z(j)r^j(1+\epsilon)^j + \sum_{j=M+1}^{\infty}z(j)r^{-j},
    \end{align*}
    which gives
    \[
        1=\frac{z(n)}{P(n)} + \frac{1}{P(n)}\sum_{j=1}^{\infty}z(j)z(n+j)
        \leq \frac{z(n)}{P(n)} +\sum_{j=1}^{M}z(j)r^j(1+\epsilon)^j + \sum_{j=M+1}^{\infty}z(j)r^{-j}.
    \]
    Thus
    \[
        \frac{z(n)}{P(n)}\geq 1 - \sum_{j=1}^{M}z(j)r^j(1+\epsilon)^j - \sum_{j=M+1}^{\infty}z(j)r^{-j},
        \quad n\geq N_0(\epsilon).
    \]
    Hence
    \[
    \liminf_{n\to\infty} \frac{z(n)}{P(n)}\geq 1 - \sum_{j=1}^{M}z(j)r^j(1+\epsilon)^j - \sum_{j=M+1}^{\infty}z(j)r^{-j}.
    \]
    Let $\epsilon\to0$ from the right, then let $M\to\infty$ to get
    \[
        \liminf_{n\to\infty}\frac{z(n)}{P(n)}\geq 1 - \sum_{j=1}^{\infty}z(j)r^j =L_1>0,
    \]
    where the fact that $L_1>0$ is a consequence of assumption \eqref{eq:bhalf}.

    The lower bound $L_1$ is used then to determine the upper bound $U_2$: we rewrite \eqref{eq:u1}
    according to
    \[
        z(n) + z(n+1)z(1) = P(n)- \sum_{j=2}^{\infty}z(j)z(n+j)\leq P(n).
    \]
    Since $\liminf_{n\to\infty} z(n)/P(n)\geq L_1$, for all $\epsilon\in (0,1)$ there exists an $N_3(\epsilon)\in\Z^+$
    such that for all $n\geq N_3(\epsilon)$
    \[
        \frac{z(n)}{P(n)}\leq 1 - z(1)\frac{P(n+1)}{P(n)}\frac{z(n+1)}{P(n+1)}\leq1-z(1)\frac{P(n+1)}{P(n)}L_1(1-\epsilon).
    \]
    Hence as $P(n+1)/P(n)\to r$ as $n\to\infty$, we get
    \[
        \limsup_{n\to\infty}\frac{z(n)}{P(n)}\leq 1-z(1)rL_1(1-\epsilon).
    \]
    Let $\epsilon\to0$ from the right to get $\limsup_{n\to\infty} z(n)/P(n)\leq 1-z(1)rL_1 = U_2$. Therefore
    we have established \eqref{eq:boundzrhom} for $m=1$.

    Regarding the induction step at level $m$ for $m\geq2$, assume that \eqref{eq:boundzrhom} holds, i.e.,
    \[
        \limsup_{n\to\infty}\frac{z(n)}{P(n)}\leq U_m, \quad \liminf_{n\to\infty}\frac{z(n)}{P(n)}\geq L_{m-1}.
    \]
    This implies that, for all $\epsilon>0$ sufficiently small, there exists $N_1(\epsilon)>0$ such that
    $z(n)/P(n)\leq U_{m}(1+\epsilon)$ for all $n\geq N_1(\epsilon)$.

    Fix $M\in\Z^+$, and let $N_0(\epsilon)$ be as defined above. Then for $n\geq\max(N_1(\epsilon),N_0(\epsilon))$,
    we note that
    \begin{align*}
        \sum_{j=1}^{\infty}z(j)\frac{z(n+j)}{P(n)} &=\sum_{j=1}^{\infty}z(j)\frac{z(n+j)}{P(n+j)}\frac{P(n+j)}{P(n)}
        \leq \sum_{j=1}^{\infty}z(j)U_{m}(1+\epsilon)\frac{P(n+j)}{P(n)} \\
        &=\sum_{j=1}^{M}z(j)U_{m}(1+\epsilon)\frac{P(n+j)}{P(n)}
        + \sum_{j=M+1}^{\infty}z(j)U_{m}(1+\epsilon)\frac{P(n+j)}{P(n)} \\
        &\leq \sum_{j=1}^{M}z(j)U_{m}(1+\epsilon)r^j(1+\epsilon)^j + \sum_{j=M+1}^{\infty}z(j)U_{m}(1+\epsilon)r^{-j} .
    \end{align*}
    Hence
    \begin{multline*}
        1=\frac{z(n)}{P(n)}+\frac{1}{P(n)}\sum_{j=1}^{\infty}z(j)z(n+j)\\
        \leq\frac{z(n)}{P(n)}+ \sum_{j=1}^{M}z(j)U_{m}(1+\epsilon)r^j(1+\epsilon)^j
        + \sum_{j=M+1}^{\infty}z(j)U_{m}(1+\epsilon)r^{-j} ,
    \end{multline*}
    which rearranges to give
    \[
        \liminf_{n\to\infty}\frac{z(n)}{P(n)}
        \geq 1- U_{m}(1+\epsilon)\left(\sum_{j=1}^{M}z(j)r^j(1+\epsilon)^j+\sum_{j=M+1}^{\infty}z(j)r^{-j}\right),
    \]
    having taken the limit inferior as $n\to\infty$.
    Letting $\epsilon\to0$ from the right, and then letting $M\to\infty$, gives
    \[
        \liminf_{n\to\infty}\frac{z(n)}{P(n)}\geq 1- U_{m}\sum_{j=1}^{\infty}z(j)r^j=L_{m}.
    \]
    This yields the lower limit in \eqref{eq:boundzrhom} at level $m+1$.

    It remains to show that the upper limit in \eqref{eq:boundzrhom} holds at level $m+1$.
    To prove this, we start by rewriting \eqref{eq:u1} in the form
    \[
        z(n)+\sum_{j=1}^{m}z(j)z(n+j) + \sum_{j=m+1}^{\infty}z(j)z(n+j) = P(n),
        \]
        which gives
        \begin{equation} \label{eq.mterms}
        \frac{z(n)}{P(n)} +\frac{1}{P(n)}\sum_{j=1}^{m}z(j)z(n+j) =1-\frac{1}{P(n)}\sum_{j=m+1}^{\infty}z(j)z(n+j)\leq 1.
        \end{equation}
    Since $\liminf_{n\to\infty}z(n)/P(n)\geq L_m$, for every $\epsilon\in(0,1)$ there is an $N_2(\epsilon)\in\Z^+$
    such that $n\geq N_2(\epsilon)$ implies $z(n)/P(n)>L_m(1-\epsilon)$.

Let $n\geq\max(N_2(\epsilon),N_0(\epsilon))$. Then
\begin{align*}
\frac{1}{P(n)}\sum_{j=1}^{m}z(j)z(n+j)
=
    \sum_{j=1}^{m}z(j)\frac{z(n+j)}{P(n+j)}\frac{P(n+j)}{P(n)}
    \geq\sum_{j=1}^{m}z(j)\frac{P(n+j)}{P(n)}L_m(1-\epsilon).
\end{align*}
Inserting this estimate into \eqref{eq.mterms} and rearranging yields
\[
    \frac{z(n)}{P(n)}\leq 1-L_m(1-\epsilon)\sum_{j=1}^{m}z(j)\frac{P(n+j)}{P(n)},
    \quad n\geq \max(N_2(\epsilon),N_0(\epsilon)).
\]
Therefore, using the positivity of $P$ and $z$, we get
\begin{align*}
    \limsup_{n\to\infty}\frac{z(n)}{P(n)}
    &\leq 1 +\limsup_{n\to\infty}\biggl(-L_m(1-\epsilon)\sum_{j=1}^{m}z(j)\frac{P(n+j)}{P(n)} \biggr) \\
    &=1-\liminf_{n\to\infty}\biggl(\sum_{j=1}^{m}z(j)\frac{P(n+j)}{P(n)}\biggr)L_m(1-\epsilon).
\end{align*}
Since $P(n+j)/P(n)\to r^j$ as $n\to\infty$, and the sum contains only finitely many terms, we have that
\[
\liminf_{n\to\infty}\biggl(\sum_{j=1}^{m}z(j)\frac{P(n+j)}{P(n)}\biggr)
=\lim_{n\to\infty}\biggl(\sum_{j=1}^{m}z(j)\frac{P(n+j)}{P(n)}\biggr)=
\sum_{j=1}^{m}z(j)r^j.
\]
Hence
\begin{align*}
\limsup_{n\to\infty}\frac{z(n)}{P(n)}
    &\leq
    1-\sum_{j=1}^{m}z(j)r^jL_m(1-\epsilon).
\end{align*}
Letting $\epsilon\to0^+$ yields
\[
    \limsup_{n\to\infty}\frac{z(n)}{P(n)}\leq 1 -\sum_{j=1}^{m}z(j)r^jL_m=U_{m+1},
\]
by the definition of $U_{m+1}$. Thus we have shown that if the $m$--th level statement in \eqref{eq:boundzrhom} holds, then
\[
    L_{m}\leq\liminf_{n\to\infty}\frac{z(n)}{P(n)}\leq\limsup_{n\to\infty}\frac{z(n)}{P(n)}\leq U_{m+1},
\]
which is the $(m+1)$--th level statement in \eqref{eq:boundzrhom}. This completes the proof of the general induction step, and
since we have already shown that \eqref{eq:boundzrhom} holds for $m=1$, the lemma is true.
\end{proof}
\begin{proof}[Proof of Theorems~\ref{thm:4a} and~\ref{thm:6a}]
The implication that $b\in\mathcal{W}^{\downarrow}(r)$ gives rise to $\rho\in\mathcal{W}^{\downarrow}(r)$, for $r\in(0,1]$ is nothing other than the subject of Theorems~\ref{thm:4} and~\ref{thm:6}. The converse result that $\rho\in\mathcal{W}^{\downarrow}(r)$ implies $b\in\mathcal{W}^{\downarrow}(r)$, for $r\in(0,1]$, is an immediate consequence
of Remark~\ref{rk:l1} and Lemmas~\ref{lm:bres},~\ref{lm:udlim} and
~\ref{lm:resb} with \eqref{eq:nsstat} being required to guarantee that $\mathbb{E}[\nu(0)^2]$ is well defined and finite.

It can be seen that the sequence $U_m$ and $L_m$ have the same limit as $m\to\infty$. By virtue of Lemma~\ref{lm:udlim},
we may take the limit as $m\to\infty$ on both sides of \eqref{eq:boundzrhom}, which yields
$\lim_{n\to\infty} z(n)/P(n)=\lim_{m\to\infty} L_m=\lim_{m\to\infty} U_{m+1}$,
from which the result follows.
\end{proof}
\subsection{Bounds}
The proof of Theorem~\ref{thm:bounds} uses a result concerning the
boundedness of linear Volterra operators in~\cite[Theorem~5.1]{jaigdr:2006}. We state a scalar variant of this theorem.
Consider the non--convolution linear Volterra summation equation
\begin{equation}\label{eq:3a}
    z(n+1)=\sum_{i=0}^{n}H(n,i)z(i), \quad n\in\Z^+;
\end{equation}
where $z(0)=z_0\in\R$ and $H:\Z^+\times\Z^+\to\R$ with $H(n,i)=0$ for $i>n$.
\begin{lemma}\label{thm:Vtabnd}
Suppose that there are integers $M$ and $N$ with $0<M<N$ such that
\begin{align*}
    \sup_{n\geq N}\sum_{i=M}^{n}|H(n,i)|<1, \quad \sup_{n\geq M}\sum_{i=0}^{M}|H(n,i)|<+\infty.
\end{align*}
Then there is $K>0$ independent of $z_0$ such that the solution of equations \eqref{eq:3a} satisfies
$|z(n)|\leq K |z_0|$ for $n\geq0$.
\end{lemma}

\begin{proof}[Proof of Theorem~\ref{thm:bounds}]
    We deal here only with the case $r=1$. The case $r<1$ follows the same steps as that of $r=1$.
    We firstly show that $z/\gamma$ is bounded. In order to write \eqref{eq:deltadiff2} as a
    convolution equation we define $\beta(n)=\l b(n+1)$. Thus
$\beta(n)\leq C_0\gamma(n)$ for some $C_0>0$ and all $n$. Then defining
$x=z/\gamma$ and using \eqref{eq:deltadiff2}, we have
    \[
        x(n+1) = \sum_{j=0}^{n}H(n,j)x(j), \quad n\geq0, \quad x(0)=1/\g(0),
    \]
    where
    \[
    H(n,j):=\frac{\beta(n-j)\g(j)}{\g(n)}\frac{\g(n)}{\g(n+1)}, \quad n\geq j\geq 0.
    \]
    To show the boundedness of $x$ we apply Lemma~\ref{thm:Vtabnd}. That is, we must show that
    \[
     W_H:=   \lim_{N\to\infty}\limsup_{n\to\infty}\sum_{j=N}^{n}H(n,j)<1
    \]
    and $H_M:=\sup_{n\geq M}\sum_{j=0}^{M}H(n,j)$ is finite for each $M\in\Z^{+}$. By the definition of $H$ and \eqref{eq:p2}
    we get
    \[
        \limsup_{n\to\infty}\sum_{j=N}^{n}H(n,j) = \limsup_{n\to\infty}\sum_{j=N}^{n}\frac{\b(n-j)\g(j)}{\g(n)}.
    \]
    Let $n\geq2N$. Then
    \begin{align*}
        \sum_{j=N}^{n}\frac{\b(n-j)\g(j)}{\g(n)} &= \sum_{l=0}^{n-N}\b(l)\frac{\g(n-l)}{\g(n)} 
     \leq  \sum_{l=0}^{N-1}\b(l)\frac{\g(n-l)}{\g(n)} + C_0\sum_{l=N}^{n-N}\frac{\g(l)\g(n-l)}{\g(n)}.
    \end{align*}
    Thus by \eqref{eq:p2}
    \[
        \limsup_{n\to\infty}\sum_{j=N}^{n}H(n,j) \leq \sum_{l=0}^{N-1}\b(l)
        + C_0\limsup_{n\to\infty}\sum_{l=N}^{n-N}\frac{\g(l)\g(n-l)}{\g(n)},
    \]
    and by \eqref{eq:p1} we get
    \begin{align*}
        W_{H} &= \lim_{N\to\infty}\limsup_{n\to\infty}\sum_{j=N}^{n}H(n,j)\\
       & \leq \sum_{l=0}^{\infty}\b(l) + C_0\lim_{N\to\infty}\limsup_{n\to\infty}\sum_{l=N}^{n-N}\frac{\g(l)\g(n-l)}{\g(n)}
=\sum_{l=0}^\infty \beta(l),
    \end{align*}
    so $W_H<1$ as required.
    Now to show that for each fixed $M$, $H_M$ is bounded, we note for $n\geq M$ that
    \begin{align*}
        \sum_{j=0}^{M}H(n,j) &= \sum_{j=0}^{M}\frac{\b(n-j)}{\g(n-j)}\frac{\g(j)\g(n-j)}{\g(n)}\frac{\g(n)}{\g(n+1)} \\
        &\leq C_0\sup_{n\geq0}\left(\frac{\g(n)}{\g(n+1)}\right)\sum_{j=0}^{M}\frac{\g(j)\g(n-j)}{\g(n)} \\
        &\leq C_0\sup_{n\geq0}\left(\frac{\g(n)}{\g(n+1)}\right)\sup_{n\geq M}\left(\frac{(\g*\g)(n)}{\g(n)}\right)
    \end{align*}
    and so $\sup_{n\geq M}H_M(n)$ is finite and therefore $x$ is bounded. As a bound on the resolvent is established,
 it just remains to deduce the bound on the autocovariance
function. Moreover, it is immediate from $x(n)=z(n)/\g(n)\leq C_1$ that
$z$ is summable. Hence
    \begin{align*}
        \rho(n) &= G\sum_{j=0}^{\infty}z(j)z(n+j)
        \leq  G C_1\sum_{j=0}^{\infty}z(j)\frac{\g(n+j)}{\g(n)}\g(n) \leq  G C_1\g(n)\sum_{j=0}^{\infty}z(j),
    \end{align*}
    and the desired result holds, where $G=\mathbb{E}[\nu(0)^2]$.
\end{proof}
\begin{proof}[Proof of Theorem~\ref{thm:subsup}]
    First let us suppose that $\limsup_{n\to\infty}b(n)/\gamma(n)=:L_3\in(0,\infty)$. Then from \eqref{eq:nont},
    \[
        \limsup_{n\to\infty}\frac{\rho(n)}{\gamma(n)}\geq \l \rho(j^*) r^{j^*} L_3 >0,
    \]
    where $j^\ast$ is the integer introduced in \eqref{A1}.     Furthermore, for any fixed $\epsilon>0$ there exists an $N(\epsilon)\in\Z^+$ such that
    $b(n)<L_3 (1+\epsilon)\gamma(n)$ for all $n\geq N(\epsilon)$. Moreover, $b(n)\leq C_{\epsilon} \gamma(n)$   for all $n\geq1$, where $C_{\epsilon}=\max\{L_3(1+\epsilon),\sup_{1\leq j \leq N(\epsilon)}b(j)/\gamma(j)\}$. Therefore, from Theorem~\ref{thm:bounds} we have that there exists $C_{1,\epsilon}>0$ such that $\rho(n)\leq C_{1,\epsilon} \gamma(n)$    for all $n\geq1$. Thus,
    \[
        0<\l \rho(j^*)L_3\leq \limsup_{n\to\infty}\frac{\rho(n)}{\gamma(n)} \leq  C_{1,\epsilon}<\infty.
    \]
    Conversely, suppose now that $\limsup_{n\to\infty}\rho(n)/\gamma(n)=:L_2\in(0,\infty)$. Then from \eqref{eq:nont} we have $\limsup_{n\to\infty} b(n)/\gamma(n) \leq L_2/(\l\rho(j^*) r^{j^*})<+\infty$.

    To show that $\limsup_{n\to\infty}b(n)/\gamma(n)>0$, we suppose the contrary, namely that $\limsup_{n\to\infty} b(n)/\gamma(n)=0$.  Since $b$ and $\gamma$ are non--negative, $\lim_{n\to\infty}b(n)/\gamma(n)=0$. Then it is not difficult to see from the proof of Theorem~\ref{thm:4} that $\lim_{n\to\infty}\rho(n)/\gamma(n)=0$ and hence $\limsup_{n\to\infty}\rho(n)/\gamma(n)=0$, which contradicts  $\limsup_{n\to\infty}\rho(n)/\gamma(n)>0$. Therefore, as
$\limsup_{n\to\infty} b(n)/\gamma(n)$ must exist, we have
    $\limsup_{n\to\infty} b(n)/\gamma(n)\in(0,\infty)$.
\end{proof}
\begin{proof}[Proof of Theorem~\ref{thm:11}]
    The proof is largely established by rewriting the limits in terms of their $\epsilon-N$ definition.
    This delivers upper and lower bounds, $\gamma_-,\gamma_+$ respectively, on $b$ where $\gamma_-(n)=C_-(n+1)^{-\beta(1-\epsilon)}$
    and  $\gamma_+(n)=C_+(n+1)^{-\beta(1+\epsilon)}$ for $n\geq0$ and for some constants $C_-,C_+>0$.
    Theorems~\ref{thm:bounds},~\ref{thm:conub} and~\ref{thm:lowb} are then applied to generate the
    appropriate bounds on $\rho$, from which the result follows.

    In order to establish \eqref{eq:lnspb}, i.e.
    \[
        \limsup_{n\to\infty}\frac{\log \rho(n)}{\log n} =-\beta \text{ implies }    \limsup_{n\to\infty}\frac{\log b(n)}{\log n} =-\beta,
    \]
    one uses \eqref{eq:nont} and an argument by contradiction, not unlike that employed in the proof of Theorem~\ref{thm:subsup}.

    For the case $\beta=\infty$, the bounding function is $n^{-K}$ where $K>0$ can be chosen arbitrarily large. In all other respects  this case follows through as for other values of $\beta$.
\end{proof}
\begin{proof}[Proof of Theorem~\ref{thm:22}]
Firstly suppose $\rho(k)\leq C_2\alpha_2^k$. By definition, $b\geq0$ and hence $z\geq0$ and $\rho\geq0$.
Thus with $j^*$ as defined in \eqref{A1}, from \eqref{eq:nont} we have
\[
    b(k+1+j^*) \leq \frac{1}{\l\rho(j^*)} \rho(k+1) \leq \frac{C_2}{\l\rho(j^*)}\alpha_2^{k+1}
    = \frac{C_2}{\l\rho(j^*)\alpha_2^{j^*}} \alpha_2^{k+1+j^*}.
\]
Hence, $b(k)\leq C_3\alpha_2^k$ for all $k\geq j^*+1$ where $C_3=C_2/(\l\rho(j^*)\alpha_2^{j^*})$ and so
$b(k)\leq C_4\alpha_2^k$ for all $k\geq1$, where $C_4=\max(C_3, Q)$ and $Q=\max_{1\leq l\leq j^*}{b(l)\alpha_2^{-l}}=b(j^*)\alpha_2^{-j^*}$.

Conversely, suppose that $b(k)\leq C_1\alpha_1^k$.
As \eqref{eq:con1} holds we have $z(n)\to0$, as $n\to\infty$. Thus
we may use~\cite[Theorem~4]{ElMur:1996a} to conclude that
\begin{equation}\label{eq:resexp}
    b(k)\leq C_1\alpha_1^k \quad \text{ if and only if } \quad z(k)\leq C_4\alpha_4^k,
\end{equation}
for some $\alpha_4\in(0,1)$ and $C_1,C_4\in(0,\infty)$.
Therefore for the sequence $f$ given in \eqref{eq:cov2}, we get
\begin{align*}
    f(k) = \l\sum_{j=1}^{\infty}b(k+j+1)\rho(-j) \leq \l C_1\sum_{j=1}^{\infty}\alpha_1^{k+j+1}\rho(j)
    <\l C_1\alpha_1\alpha_1^k\sum_{j=1}^{\infty}\rho(j).
    \end{align*}
    Thus as $\rho$ is summable from Theorem~\ref{thm:lgds3}, we have
        $f(k)\leq \l C_1K\alpha_1^{k},$
for some $0<K<\infty$.
Using this estimate for $f$ and \eqref{eq:resexp} in \eqref{eq:varpar} gives
\begin{equation}
    \rho(k) \leq C_5\alpha_4^k+\sum_{j=1}^{k}C_4\alpha_4^{k-j}C_6\alpha_1^j
    = C_5\alpha_4^k+C_7\alpha_4^k\sum_{j=1}^{k}\left(\frac{\alpha_1}{\alpha_4}\right)^j.
\end{equation}
If $\alpha_1\not=\alpha_4$, with $\alpha_2=$max$(\alpha_1,\alpha_4)$ we have
    $\rho(k) \leq C_5\alpha_4^k+C_8|\alpha_4^k-\alpha_1^k|
    \leq C_5\alpha_4^k+C_8\alpha_4^k+C_8\alpha_1^k
    \leq C_9\alpha_2^k$. If $\alpha_1=\alpha_4$, then
\[
    \rho(k) \leq C_5\alpha_4^k+C_7\alpha_4^kk < C_5\alpha_4^k+C_7C_8(\alpha_4+\epsilon)^k<
    C_{10}(\alpha_4+\epsilon)^k,
\]
where $\alpha_2=\alpha_4+\epsilon$ and $\epsilon$ is chosen sufficiently small so that  $\alpha_2<1$, and
$C_8$ is given by $C_8 = \sup_{k\geq1} k/(1+\epsilon/\alpha_4)^k$.
\end{proof}

\begin{flushleft}
\textbf{Acknowledgements}
\end{flushleft}The authors are grateful to David Reynolds for his advice
and scrutiny of the article. In particular, they thank him for stimulating conversations in relation to the sufficient conditions for stationary.

\providecommand{\bysame}{\leavevmode\hbox
to3em{\hrulefill}\thinspace}
\providecommand{\MR}{\relax\ifhmode\unskip\space\fi MR }
\providecommand{\MRhref}[2]{
  \href{http://www.ams.org/mathscinet-getitem?mr=#1}{#2}
} \providecommand{\href}[2]{#2}


\end{document}